\documentclass{imanum}
\usepackage{graphicx} 

\received{12 January 2026}

\usepackage{mathtools}
\usepackage{amsthm,amssymb}
\usepackage{mathrsfs}
\usepackage[hidelinks]{hyperref}
\usepackage{adjustbox}

\def\de{{\mathrm{d}}}

\begin{document}

\title{Scattered Data Histopolation in Averaging Kernel Hilbert Spaces}
\shorttitle{Scattered Data Histopolation in Averaging Kernel Hilbert Spaces}

\author{%
{\sc
Ludovico Bruni Bruno \thanks{Email: ludovico.brunibruno@polito.it}
} \\[2pt]
Dipartimento di Scienze Matematiche
"Giuseppe Luigi Lagrange"
\\Politecnico di Torino, Corso Duca degli Abruzzi 24, 10129 Torino, Italy\\[6pt]
{\sc and}\\[6pt]
{\sc Giacomo Cappellazzo \thanks{Email: giacomo.cappellazzo@math.unipd.it} and Wolfgang Erb \thanks{Email: wolfgang.erb@unipd.it (corresponding author)}}\\[2pt]
Dipartimento di Matematica "Tullio Levi-Civita"
\\Universit{\`a} degli Studi di Padova, Via Trieste 63, 35121 Padova, Italy\\[6pt]
{\sc and}\\[6pt]
{\sc Mohammad Karimnejad Esfahani \thanks{Email: m.karimnejad7395@gmail.com}}\\[2pt]
 Dipartimento di Matematica \\ Universit{\`a} degli Studi di Genova, 
 Via Dodecaneso 35, 16146, Genova, Italy 
}

\shortauthorlist{L. Bruni Bruno, G. Cappellazzo, W. Erb, M. Karimnejad Esfahani}

\maketitle

\begin{abstract}
{Kernel-based methods offer a powerful and flexible mathematical framework for addressing  histopolation problems. In histopolation, the available input data does not consist of pointwise function samples but of averages taken over intervals or higher-dimensional regions, and these mean values serve as a basis for reconstructing or approximating the target function. While classical interpolation requires continuity of the underlying function, histopolation can be performed in larger function spaces. In the framework of kernel methods, we will introduce and study the so-called averaging kernel Hilbert spaces (AKHS's) for this purpose. Within this setting, we develop systematic construction principles for averaging kernels and provide characterizations based on the Fourier-Plancherel transform. In addition, we analyze several representative histopolation scenarios in order to highlight properties of this approximation method, including conditions for unisolvence and possible error estimates. Finally, we present numerical experiments that shed some light on the convergence behavior of the presented approach and demonstrate its practical effectiveness.}{Histopolation of scattered average values; Kernel-based approximation; Averaging Kernel Hilbert Spaces; Reproducing Kernel Hilbert Spaces.} 
\end{abstract}

\section{Introduction}

Histopolation refers to approximation methods in which the mean values of functions over domains, such as intervals, cells, or higher-dimensional regions, are used as initial data rather than pointwise function samples used in interpolation. The term histopolation itself is a portmanteau formed from histogram and interpolation, and has been studied intensively in the context of splines, cf. \citep{Siewer2008,Schoenberg1973,Subbotin1977} or 
\citep[Chapter VIII]{deBoor2001}, in which spline approximants have to satisfy given area matching conditions. Using alternative basis functions, this idea was also studied for rational splines \citep{FischerOja2007}, periodic splines \citep{Delvos1987}, or polynomials \citep{BruniAccioErbNudo2025,BruniErb,BruniErb2025,robidoux2006}.  
 
Also kernel-based methods have been used to solve histopolation-type problems. Compared to polynomials, kernel approximants turn out to be more stable and robust, and offer more flexibility when it comes to handle unstructured data sets. This was, for instance, observed for finite volume methods and in ENO/WENO schemes \citep{AboiyarGeorgoulisIske2010,IskeSonar1996Structure,SonarFV,SonarENO,SonarENO2,WendlandFiniteVolumes}. 
Another important field of application for kernel-based histopolation is the Radon transform, where the given initial data corresponds to integral values over entire lines or manifolds rather than intervals \citep{DeMarchiIskeSironi2016,DeMarchiIskeSantin2018,Sironi2011,vanDenBoogaart2007}. More generally, kernel-based histopolation arises as a special case in the context of generalized Hermite-Birkhoff interpolation, as for instance worked out in \citep{Albrecht2024,AlbrechtIske2025,Iske1995,Wu1992,WendlandStokes} or \citep[Chapter 16]{Wendland_2004}. 

\begin{figure}[htbp]
    \centering
    \includegraphics[width=0.34\linewidth]{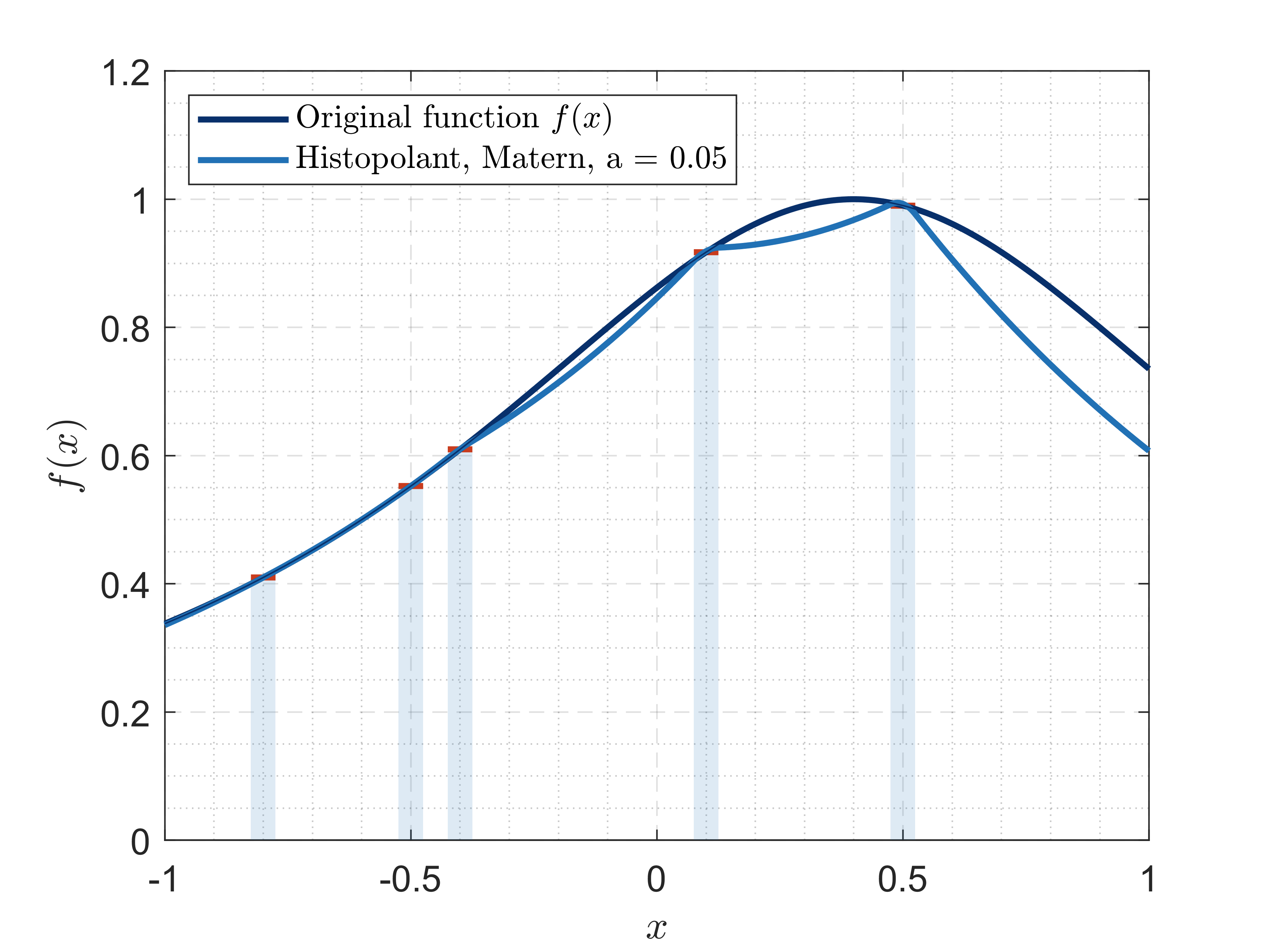} \hspace{-5mm}
    \includegraphics[width=0.34\linewidth]{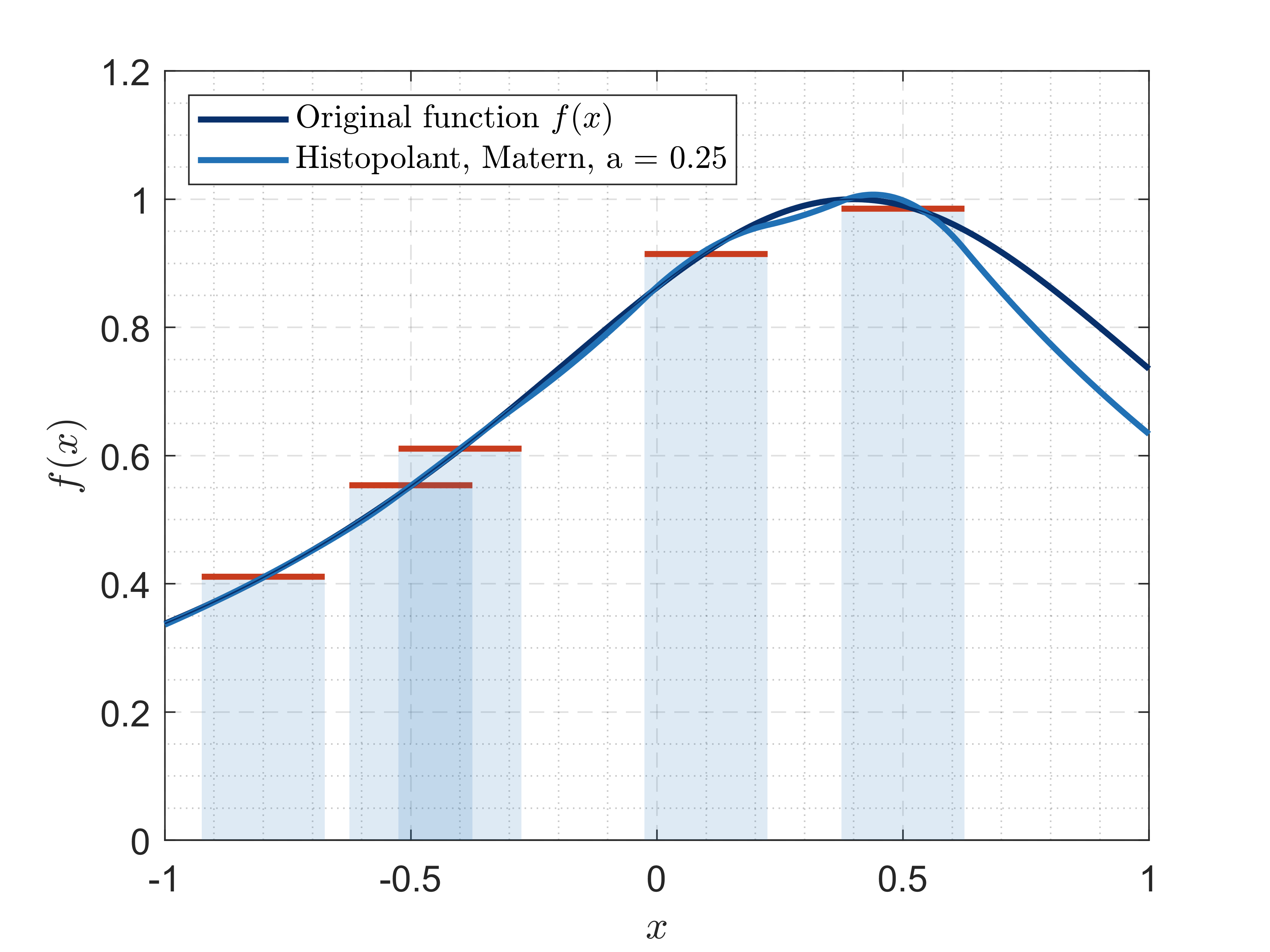} \hspace{-5mm}
    \includegraphics[width=0.34\linewidth]{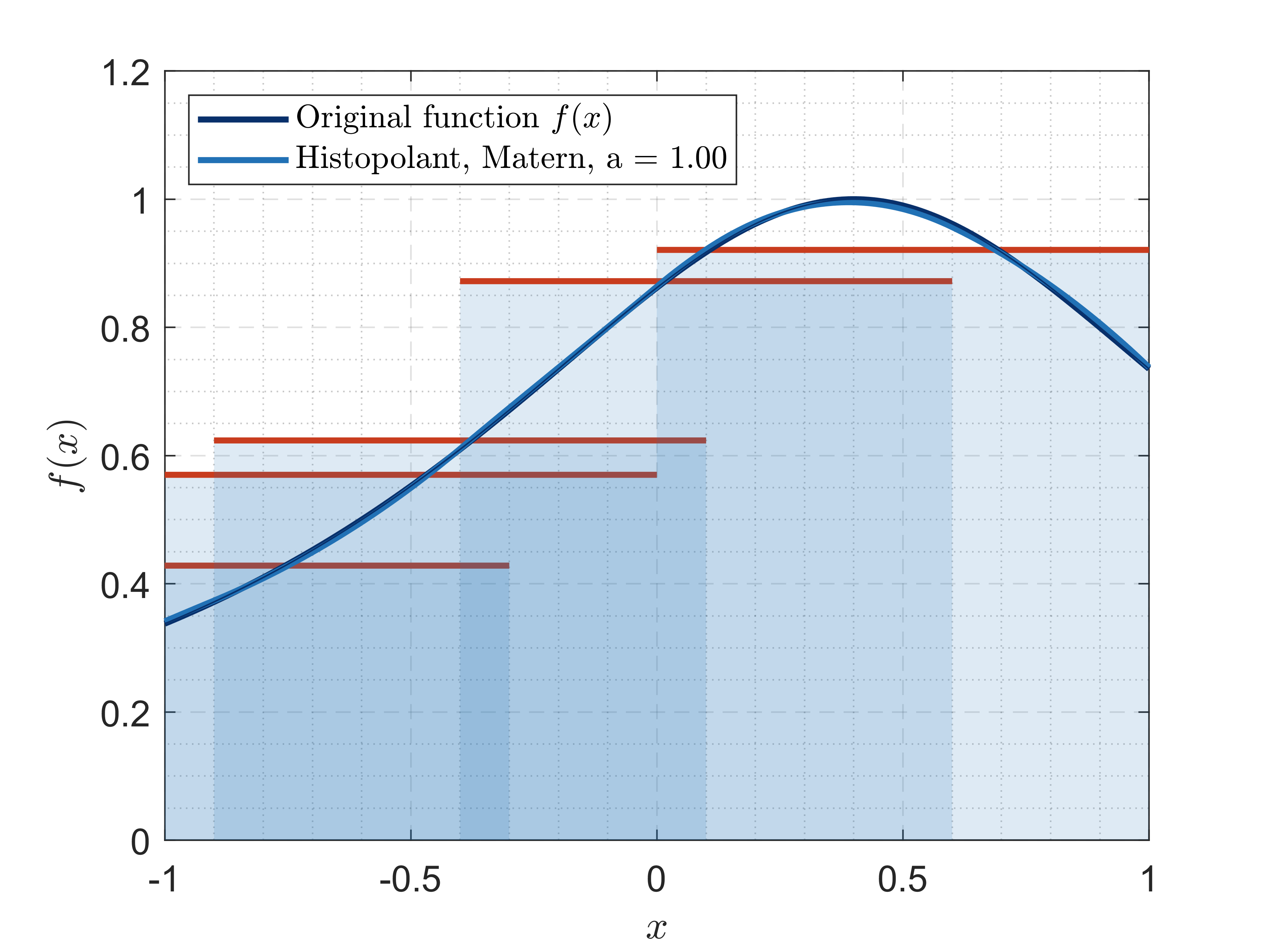}
    \caption{Histopolation of average values of the function $f(x) = \frac{1}{1+(x-0.4)^2}$ over $5$ interval segments with length $a \in \{0.05, 0.25, 1\}$ using shifts of averaged Mat\'{e}rn functions as kernels, see Example \ref{ex:RUAK} (i). }
    \label{fig:histopolationmatern}
\end{figure}

Common to all kernel-based methods is the use of an underlying reproducing kernel Hilbert space (RKHS), sometimes also referred to as native space, which serves as the fundamental approximation space for formulating (generalized) interpolation problems. These spaces contain a reproducing kernel that enables to represent point-evaluation functionals as elements of the Hilbert space itself. If the reproducing kernel is continuous these spaces are subspaces of continuous functions. On the other hand, it is known that a RKHS cannot contain all continuous functions on an uncountable compact domain \citep{steinwart2024}. By leaving the Hilbert space setting and turning to reproducing kernel Banach spaces, it becomes nevertheless possible to expand the respective approximation spaces \citep{ehring2025,lin2022}. 

In histopolation, point evaluations are not essential. Instead, it suffices that averages over the considered subdomains are continuous functionals in the underlying approximation space, a substantially weaker assumption. To obtain a theoretical foundation for histopolation, we will therefore introduce and systematically study averaging kernel Hilbert spaces (AKHS's), in which only the specified averaging functionals are required to be continuous. This new framework enables the approximation of more general, also non-continuous functions. A representative example is the space $L_2(\mathbb{R}^d)$ of square-integrable functions. Although it is not a RKHS, it can be interpreted as an AKHS. In this sense, AKHS's extend naturally the classical native-space theory and allow to leave the pure RKHS framework while still remaining within a Hilbert space setting.  

Nevertheless, AKHS's and RKHS's are closely related. In fact, we will show in Section \ref{sec:AKHS} that there exists a natural isometric isomorphism between an AKHS and an associated RKHS, enabling histopolation to be reformulated as interpolation problem in RKHS's. This relation between an AKHS and its associated RKHS has deep theoretical and practical implications for histopolation: it provides a way to characterize averaging kernels through known descriptions of reproducing kernels (Section \ref{sec:Fourier}), it allows histopolation matrices to be stated explicitly in terms of the reproducing kernel, and it yields criteria for unisolvence based on the corresponding criteria for the reproducing kernel (Section \ref{sec:histopolation}).

\begin{figure}[htbp]
    \centering
    \includegraphics[width=0.34\linewidth]{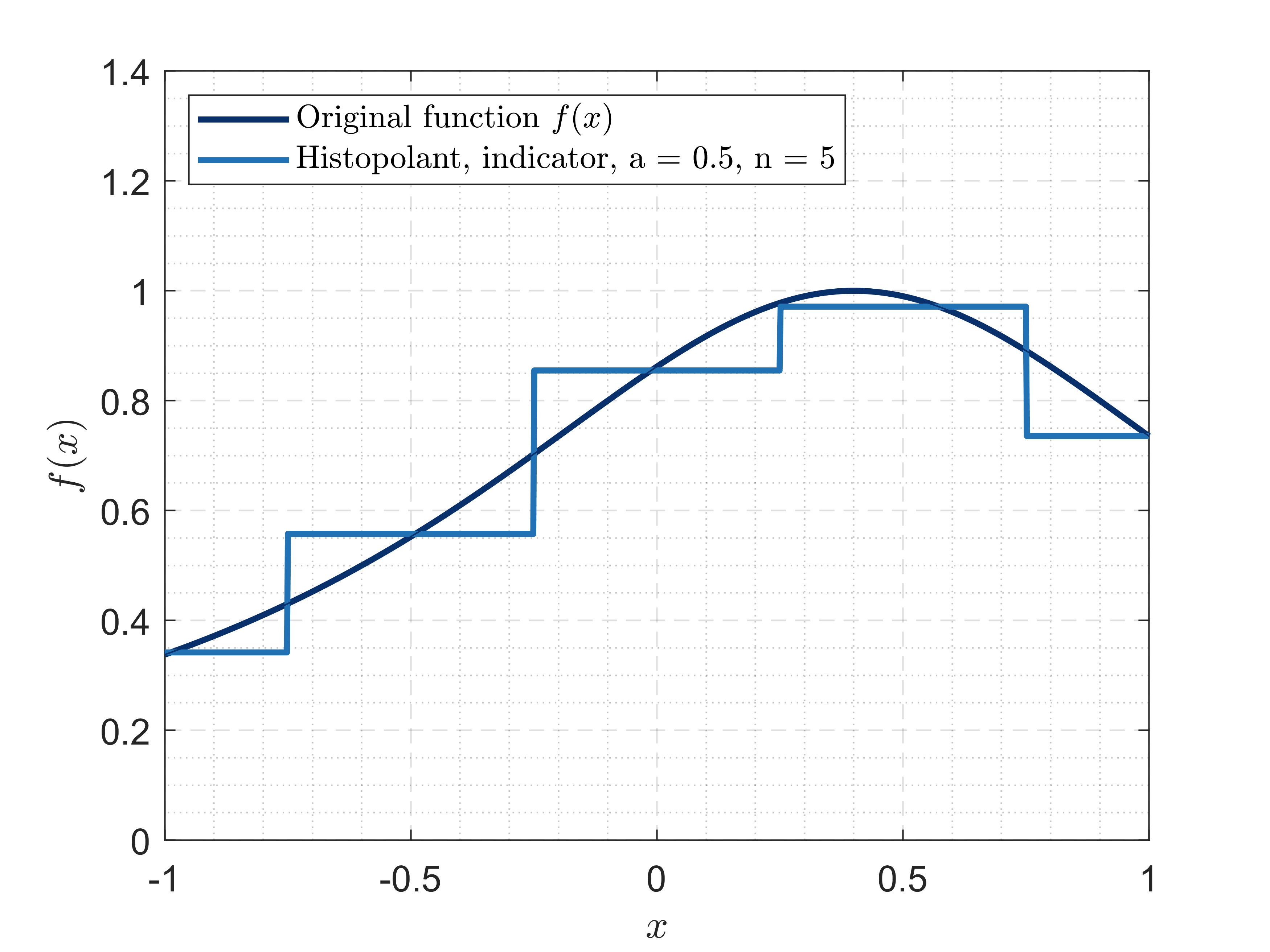} \hspace{-5mm}
    \includegraphics[width=0.34\linewidth]{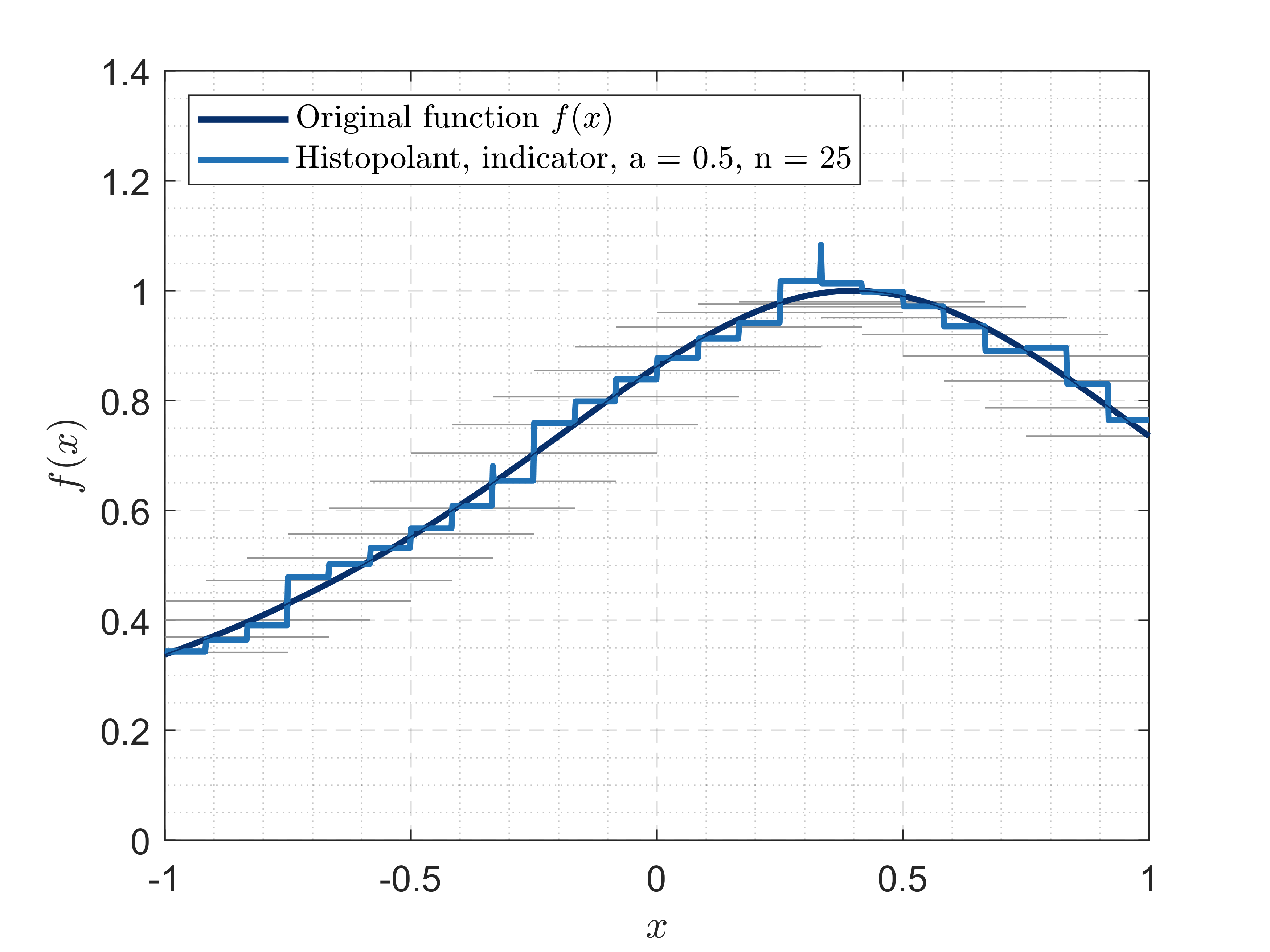} \hspace{-5mm}
    \includegraphics[width=0.34\linewidth]{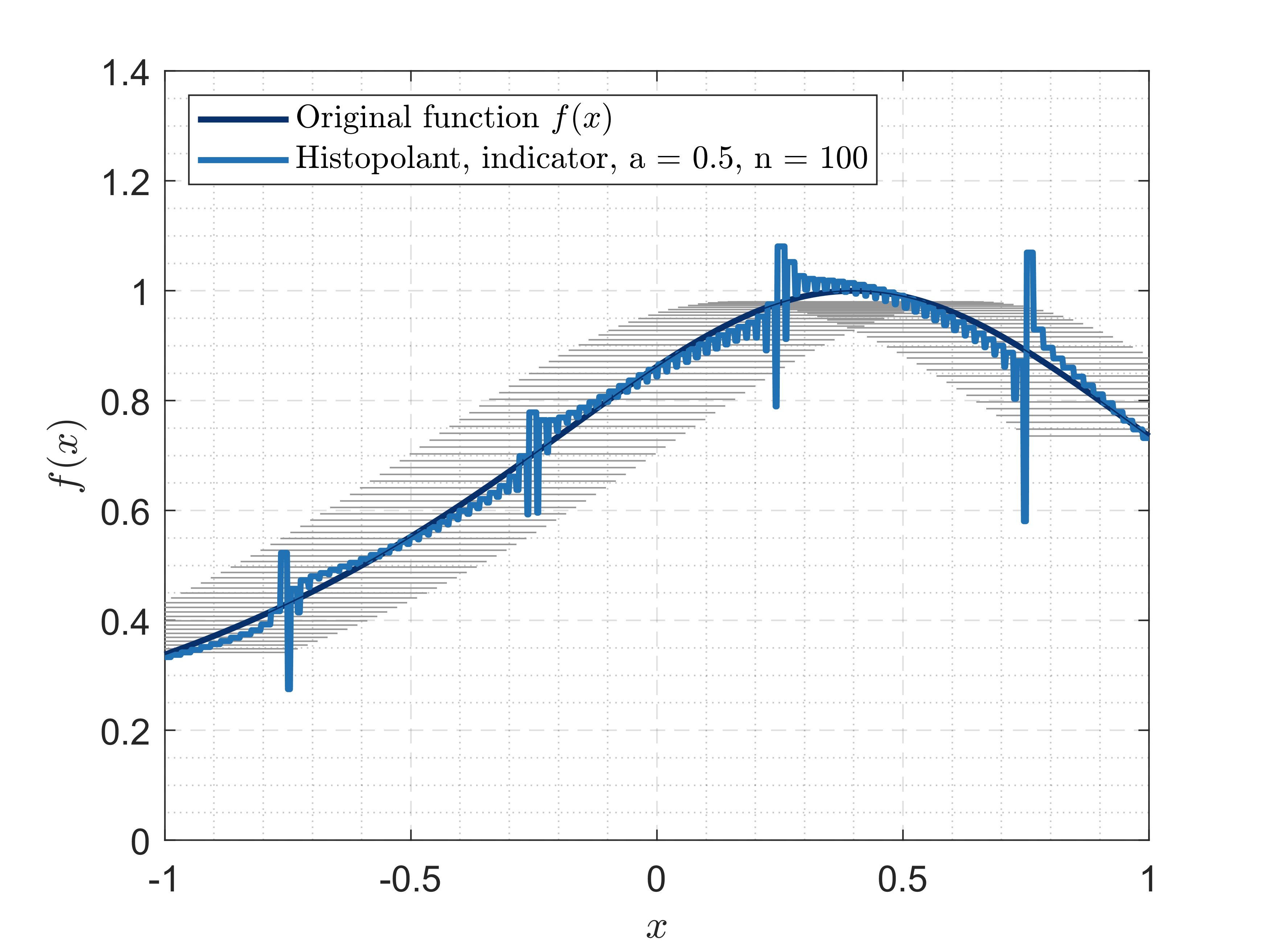}
    \caption{Histopolation of average values of the function $f(x) = \frac{1}{1+(x-0.4)^2}$ over $n \in \{5,25,100\}$ uniform intervals of length $a = 0.5$ using the indicator functions in \eqref{eq:indicatorkernel1D} as averaging kernels. }
    \label{fig:histopolationindicator}
\end{figure}

Once unisolvence is established, questions related to the convergence of histopolation towards a given exact function become significant. These questions turn out to be quite subtle. On the one hand, mean convergence for approximation in AKHS's can be derived from respective convergence estimates for reproducing kernels. More precisely, in Section \ref{sec:errorestimates} we obtain convergence criteria for the mean values of the histopolants, again as a consequence of the isomorphism between the AKHS and its associated RKHS. On the other hand, uniform convergence cannot be guaranteed without imposing additional assumptions on the subdomains. In fact, the numerical experiments in Section \ref{sec:numericalexperiments} will show that uniform convergence does not occur if intervals with constant length are chosen as averaging domains, even though the example in Fig. \ref{fig:histopolationmatern} visually suggests a good approximation accuracy if the averaged Mat{\'e}rn kernel is applied. Similar phenomena emerge for non-continuous averaging kernels based on indicator functions, such as the ones shown in Fig. \ref{fig:histopolationindicator}: uniform convergence of the corresponding histopolants seems not possible in this case, whereas convergence in the mean can still be achieved. 

To enhance the applicability of the theory on AKHS's,  
we will put considerable attention on the explicit construction of averaging kernels and their characterizations. In Section \ref{sect:exampleAKHS}, we present several examples and explicit formulas showing how to construct averaging kernels and their associated reproducing kernels. Particular focus is given to non-continuous kernels and to kernels built upon the averages of existing kernels. Likewise, in Section \ref{sec:Fourier}, we investigate AKHS's in which averages are taken over uniform domains and characterize these spaces via the Fourier-Plancherel transform. In addition, Section \ref{sect:implementationstrategy} explains how averaging kernels can be approximated using quadrature formulas when the underlying domains have a more complex topology, or when there is a lack of closed formulas. Remarkably, many relevant properties, such as symmetry and positive definiteness of the histopolation matrix, can be carried over to this situation as well. In the final Section \ref{sec:numericalexperiments}, in addition to the aforementioned experiments aimed at numerically studying the convergence behavior of kernel-based histopolation, we show how these methods can be applied to image processing tasks, such as image compression through pixel binning and image upscaling.  

\section{Averaging Kernel Hilbert Spaces} \label{sec:AKHS}

In the following, we assume that $\Omega $ 
is a subset of $\mathbb{R}^d$ and $\mathcal{H}$ is a Hilbert space of real-valued functions on $\Omega$, endowed with pointwise addition and scalar multiplication of the functions. The relative inner product is denoted by $\langle \cdot, \cdot \rangle_{\mathcal{H}}$. 

We consider a collection $\{\omega_{\tau} \ : \ \tau \in \mathcal{T}\}$ of Lebesgue measurable subsets in $\Omega$, where $\mathcal{T}$ is an arbitrary set of indices. For the Hilbert space $\mathcal{H}$, we then consider the linear \emph{averaging functionals}
\[
\lambda_\tau : f \mapsto \frac{1}{|\omega_\tau|}\int_{\omega_\tau} f(x) \mathrm{d}x, \quad \text{for all} \; f \in \mathcal{H}, \; \tau \in \mathcal{T},
\]
where $ | \omega_\tau | $ denotes the Lebesgue measure of the set $ \omega_\tau $, and assume that they are continuous on $\mathcal{H}$.
These functionals provide the averages of a function $f$ over the subsets $\omega_\tau$. 
As the functionals $\lambda_\tau$ are continuous, the Riesz representation theorem ensures that we can represent them as dual elements $A(\cdot, \tau) \in \mathcal{H}$ in the Hilbert space such that
\[\lambda_{\tau}(f) = \langle f, A(\cdot, \tau)\rangle_{\mathcal{H}}. \]
For obvious reasons, we refer to the kernel function $A :\Omega \times \mathcal{T} \to \mathbb{R}$ as \emph{averaging kernel}. Further, we call the subspace 
\[ \mathcal{N}_A(\Omega) := \overline{\mathrm{span} \{A(\cdot, \tau) \ : \ \tau \in \mathcal{T}\}} \subseteq \mathcal{H}\]
an \emph{averaging kernel Hilbert space (AKHS)}. In particular, in this Hilbert space the following averaging property holds true: 
\[
\frac{1}{|\omega_\tau|} \int_{\omega_\tau} f(x) \mathrm{d} x = \lambda_\tau(f) = \langle f, A(\cdot, \tau) \rangle_{\mathcal{H}} \quad \text{for all} \; f \in \mathcal{N}_A(\Omega).
\]
If we regard $A(\cdot,\tau)$ as a function in $\mathcal{N}_A(\Omega)$, we can as well calculate its averages and get
\[
\frac{1}{|\omega_\rho|}\int_{\omega_\rho} A(y,\tau) \mathrm{d}y = \lambda_\rho(A(\cdot,\tau)) = \langle A(\cdot,\tau), A(\cdot,\rho) \rangle_{\mathcal{H}},
\]
where $A(\cdot,\rho) \in \mathcal{N}_A(\Omega)$ is the element in $\mathcal{N}_A(\Omega)$ associated with the averaging functional $\lambda_\rho$.
This allows us to define a new  
associated kernel $K: \mathcal{T} \times \mathcal{T} \to \mathbb{R}$ by
\[
K(\tau, \rho) := \langle A(\cdot,\tau), A(\cdot,\rho) \rangle_{\mathcal{H}}.
\]
Based on this definition, it is easy to see that $K: \mathcal{T} \times \mathcal{T} \to \mathbb{R}$ is a symmetric and positive semi-definite kernel on the index set $\mathcal{T}$. Specifically, for every $n \in \mathbb{N}$, $\tau_1, \ldots, \tau_n \in \mathcal{T}$, and $c_1, \ldots, c_n \in \mathbb{R}$, we have:
\[
\sum_{i,j=1}^n c_i c_j K(\tau_i, \tau_j) = \left\|\sum_{i=1}^n c_i A(\cdot,\tau_i)\right\|_{\mathcal{H}}^2 \geq 0.
\]
Further, if the functions $A(\cdot, \tau)$ are linearly independent, the kernel $K$ is positive definite, implying a strict inequality in the latter formula if the coefficients $c_i$ are not vanishing simultaneously. 
We can associate to the kernel $K$ a reproducing kernel Hilbert space $\mathcal{N}_K(\mathcal{T})$ that is given as the completion 
\[ \mathcal{N}_K(\mathcal{T}) := \overline{\mathrm{span} \{K(\cdot, \tau) \ : \ \tau \in \mathcal{T}\}}, \]
in which the inner product of two pre-Hilbert space elements written as $f(\tau) = \sum_{i = 1}^n c_i K(\tau, \tau_i)$ and $g(\tau) = \sum_{i = 1}^n d_i K(\tau, \tau_i)$ is defined as
\[ \langle f, g \rangle_{\mathcal{N}_{K}(\mathcal{T})} := \sum_{i,j =1}^n c_i d_j K(\tau_i,\tau_j).\]

We then obtain the following isometry, which is analogous to those established in the context of Hermite-Birkhoff interpolation inside RKHS's \citep[see][Theorem 16.8]{Wendland_2004}. Novel in our setting is that within the framework of AKHS's we are no longer restricted to start in an initial RKHS. 

\begin{theorem} \label{thm:AKHSmain}
Assume that the functionals $\{ \lambda_{\tau} \ : \ \tau \in \mathcal{T}\}$ are continuous in the space $\mathcal{H}$. Then, the linear operator $\Lambda : \mathcal{N}_A(\Omega) \to \mathcal{N}_K(\mathcal{T})$ given as $(\Lambda f)(\tau) = \lambda_{\tau}(f)$ provides an isometric isomorphism between the averaging kernel Hilbert space $\mathcal{N}_A(\Omega)$ and the reproducing kernel Hilbert space $\mathcal{N}_K(\mathcal{T})$. In particular, for the respective inner products we have
\[ \langle f, g \rangle_{\mathcal{H}} = \langle \Lambda f, \Lambda g \rangle_{\mathcal{N}_K(\mathcal{T})}.\]
\end{theorem}

\begin{proof}
We show that $\Lambda$ is an isometric isomorphism between the spaces $\text{span}\{A(\cdot, \tau) \ : \ \tau \in \mathcal{T}\}$ and $\text{span} \{K(\cdot, \tau) \ : \ \tau \in \mathcal{T}\}$, and then extend the result to closures, i.e., to the relative native spaces. First, observe that the linearity of the functionals $\{ \lambda_{\tau} \ : \ \tau \in \mathcal{T}\}$ implies the linearity of $\Lambda$. Since
\begin{equation*}
    \Lambda(A(\cdot,\tau))(\rho)=\lambda_{\rho}(A(\cdot,\tau)) = \langle A(\cdot,\tau), A(\cdot,\rho)\rangle_{\mathcal{H}} = K(\tau,\rho),
\end{equation*}
we see that $\Lambda(A(\cdot,\tau))=K(\cdot,\tau)$. That $\Lambda$ is an isometric bijection now follows by simple observations:
\begin{itemize}
    \item[(i)] Isometry: the operator $\Lambda$ is an isometry since
\begin{equation*}
\begin{split}
     & \left\langle \sum_{i=1}^n \alpha_i A(\cdot,\tau_i), \sum_{j=1}^m \beta_j A(\cdot,\rho_j) \right\rangle_{\mathcal{H}}  = \sum_{i=1}^n  \sum_{j=1}^m \alpha_i \beta_j \langle A(\cdot,\tau_i), A(\cdot,\rho_j) \rangle_{\mathcal{H}} =\\
     = & \sum_{i=1}^n  \sum_{j=1}^m \alpha_i \beta_j K(\tau_i, \rho_j) = \left\langle \sum_{i=1}^n \alpha_i K(\cdot,\tau_i), \sum_{j=1}^m \beta_j K(\cdot,\rho_j) \right\rangle_{\mathcal{N}_K(\mathcal{T})}.
\end{split}
\end{equation*}
    \item[(ii)] Injectivity: if $\Lambda\left(\sum_{i=1}^n \alpha_i A(\cdot, \tau_i)\right)=0$, then, by linearity, $\sum_{i=1}^n \alpha_i K(\cdot, \tau_i)=0$. Therefore, by the isometry of the mapping $\Lambda$, we have
    \begin{equation*}
        0=\left\|\sum_{i=1}^n \alpha_i K(\cdot, \tau_i)\right\|^2_{\mathcal{N}_{K}(\mathcal{T})} = \left\|\sum_{i=1}^n \alpha_i A(\cdot, \tau_i)\right\|^2_{\mathcal{H}},
    \end{equation*}
    and this is only possible if $\sum_{i=1}^n \alpha_i A(\cdot, \tau_i)=0$. 
    \item[(iii)] Surjectivity: if $g \in \text{span} \{K(\cdot, \tau) \ : \ \tau \in \mathcal{T}\}$, then 
    \begin{equation*}
        g= \sum_{i=1}^n \alpha_i K(\cdot, \tau_i) = \sum_{i=1}^n \alpha_i \Lambda(A(\cdot, \tau_i)) = \Lambda\left(\sum_{i=1}^n \alpha_i A(\cdot, \tau_i)\right).
    \end{equation*}
\end{itemize}

By continuity, we can extend the operator $\Lambda$ to the closure of $\text{span}\{A(\cdot, \tau) : \tau \in \mathcal{T}\}$. Let $\widetilde{\Lambda}$ be the corresponding extension: if $f \in \overline{\text{span}\{A(\cdot, \tau)  \ : \ \tau \in \mathcal{T}\}}$ is in the closure, then there exists $\{f_n\}_{n \in \mathbb{N}} \subseteq \text{span}\{A(\cdot, \tau) : \tau \in \mathcal{T}\}$ such that $\lim_{n\to\infty} f_n=f$ in $\mathcal{H}$, providing the element
\begin{equation*}
    \widetilde{\Lambda}(f) := \lim_{n\to\infty} \Lambda(f_n).
\end{equation*}
Since $\mathcal{N}_K(\mathcal{T})$ is a Hilbert space, this construction is well-defined: as $\{f_n\}_{n \in \mathbb{N}}$ converges and
\begin{equation*}
    \| \Lambda(f_n) - \Lambda(f_m) \|_{\mathcal{N}_K(\mathcal{T})} = \| \Lambda(f_n -f_m) \|_{\mathcal{N}_K(\mathcal{T})} = \|f_n - f_m\|_{\mathcal{H}},
\end{equation*}
 the sequence $\{\Lambda(f_n)\}_{n \in \mathbb{N}}$ is a Cauchy sequence and therefore admits a limit in $\mathcal{N}_K(\mathcal{T})$. Moreover, the limit $\widetilde{\Lambda} f$ is independent of the specific sequence under consideration: if $\{f_n\}_{n \in \mathbb{N}}$ and $\{g_n\}_{n \in \mathbb{N}}$ are sequences in $\text{span}\{A(\cdot, \tau) : \tau \in \mathcal{T}\}$ such that $\lim_{n\to\infty} f_n=f$ and $\lim_{n\to\infty} g_n=f$, then
\begin{equation*}
\begin{split}
    & \left\|\lim_{n\to\infty} \Lambda(f_n) - \lim_{n\to\infty} \Lambda(g_n) \right\|_{\mathcal{N}_{R}(\mathcal{T})} \\
    & \quad \leq \left\|\lim_{n\to\infty} \Lambda(f_n) -\Lambda(f_m)\right\|_{\mathcal{N}_{R}(\mathcal{T})} +\|\Lambda(f_m)-\Lambda(g_m)\|_{\mathcal{N}_{R}(\mathcal{T})} + \left\|\Lambda(g_m) - \lim_{n\to\infty} \Lambda(g_n)\right\|_{\mathcal{N}_{R}(\mathcal{T})} \\
    & \quad =  \left\|\lim_{n\to\infty} \Lambda(f_n) -\Lambda(f_m)\right\|_{\mathcal{N}_{R}(\mathcal{T})} +\|f_m-g_m\|_{\mathcal{H}} + \left\|\Lambda(g_m) - \lim_{n\to\infty} \Lambda(g_n)\right\|_{\mathcal{N}_{R}(\mathcal{T})},
\end{split}
\end{equation*}
which tends to zero as $m \to \infty$.
\\
\\
We next show that also the extension $\widetilde{\Lambda}: \mathcal{N}_A(\Omega) \to \mathcal{N}_K(\mathcal{T})$ is an isometry and a bijection: 
\begin{itemize}
    \item[(i)] Isometry: for $f,g \in \mathcal{N}_A(\Omega)$ and sequences $\{f_n\}_{n \in \mathbb{N}}$, $\{g_n\}_{n \in \mathbb{N}}$  in $\text{span} \{A(\cdot,\tau): \tau \in \mathcal{T}\}$ such that $\lim_{n\to\infty} f_n =f$ and $\lim_{n\to\infty} g_n =g$, we have
    \begin{equation*}
        \langle \widetilde{\Lambda}(f),\widetilde{\Lambda}(g) \rangle_{\mathcal{N}_K(\mathcal{T})} = \lim_{n\to\infty} \langle \Lambda(f_n),\Lambda(g_n) \rangle_{\mathcal{N}_K(\mathcal{T})} = \lim_{n\to\infty} \langle f_n,g_n \rangle_{\mathcal{H}} = \langle f,g \rangle_{\mathcal{H}}.
    \end{equation*}
    since $\lim_{n\to\infty} \Lambda(f_n) =\widetilde{\Lambda}(f)$ and $\lim_{n\to\infty} \Lambda(f_n) =\widetilde{\Lambda}(f)$, and the inner product is continuous. 
    \item[(ii)] Injectivity: let $\widetilde{\Lambda}(f)=0$ for some $f \in \mathcal{N}_A(\Omega)$. Then, there exists $\{f_n\}_{n \in \mathbb{N}} \subseteq \text{span}\{A(\cdot,\tau) : \tau \in \mathcal{T}\}$ such that $\lim_{n\to\infty} f_n=f$, and
    \begin{equation*}
        \lim_{n\to\infty} \Lambda(f_n)=0 \,  \Rightarrow \lim_{n\to\infty} \|f_n\|_{\mathcal{H}} = \lim_{n\to\infty} \|\Lambda(f_n)\|_{\mathcal{N}_{R}(\mathcal{T})}=0 \, \Rightarrow \|f\|_{\mathcal{H}}=0 \,\Rightarrow \, f = 0.
    \end{equation*}
    \item[(iii)] Surjectivity: let $g \in \mathcal{N}_K(\mathcal{T})=\overline{\text{span}\{K(\cdot, \tau): \tau \in \mathcal{T}\}}$. Then, as $\Lambda(\text{span}\{A(\cdot, \tau): \tau \in \mathcal{T}\})= \text{span}\{K(\cdot, \tau): \tau \in \mathcal{T}\}$, we have a sequence $\{f_n\}_{n \in \mathbb{N}} \subseteq \text{span}\{A(\cdot, \tau): \tau \in \mathcal{T}\}$ such that $g=\lim_{n\to\infty}\Lambda(f_n)$. Now, since
    \begin{equation*}
        \|f_n-f_m\|_\mathcal{H}=\|\Lambda(f_n-f_m)\|_{\mathcal{N}_K(\mathcal{T})}=\|\Lambda(f_n)-\Lambda(f_m)\|_{\mathcal{N}_K(\mathcal{T})},
    \end{equation*}
    the sequence $\{f_n\}_{n \in \mathbb{N}}$ is a Cauchy sequence. As $\mathcal{N}_A(\Omega)$ is a Hilbert space, there exists $f \in \mathcal{N}_A(\Omega)$ such that $\lim_{n\to\infty}f_n=f$, and we can conclude that $g=\widetilde{\Lambda}(f)$.
\end{itemize}
The last step is to show that $\Lambda = \widetilde{\Lambda}$. We recall that for reproducing kernel Hilbert spaces the convergence in norm implies the pointwise convergence. If $f \in \mathcal{N}_A(\Omega)$ and $\{f_n\}_{n \in \mathbb{N}} \subseteq \text{span}\{A(\cdot,\tau): \tau \in \mathcal{T}\}$ with $\lim_{n\to\infty
}f_n=f$, then $\widetilde{\Lambda}(f)=\lim_{n\to\infty} \Lambda(f_n)$. Therefore, we get
\begin{equation*}
    \widetilde{\Lambda}(f)(\tau)= \lim_{n\to\infty} \Lambda(f_n)(\tau)= \lim_{n\to\infty} \lambda_{\tau}(f_n)= \lambda_{\tau}(f) = \Lambda(f)(\tau) \quad \text{for all} \ \tau \in \mathcal{T}.
\end{equation*}
This concludes the proof.
\end{proof}

\section{Construction and examples of AKHS} \label{sect:exampleAKHS}

In this section we provide several explicit examples of averaging kernel Hilbert spaces. In particular, we show that spaces which do not admit a reproducing kernel may nevertheless admit an averaging kernel. A particular focus is given to the construction of AKHS's from given RKHS's and on uniform AKHS's in which the averaging domains are shifts of a single domain.  

\subsection{Square integrable functions} 
As a first example, we consider the Hilbert space $\mathcal{H} = L_2(\Omega)$ 
equipped with the usual inner product $\langle f, g \rangle_{\mathcal{H}} = \int_{\Omega} f(x) g(x) \mathrm{d} x$. This Hilbert space does not have a reproducing kernel, but we can build averaging kernels for $L_2(\Omega)$. For a given set $\{\omega_{\tau} \ : \ \tau \in \mathcal{T}\}$ of measureable subdomains in $\Omega$ the indicator functions
\[ A(x,\tau) = \frac{1}{|\omega_{\tau}|}\chi_{\omega_{\tau}}(x),\]
define an averaging kernel and a respective AKHS inside $L_2(\Omega)$. The associated reproducing kernel $K$ on $\mathcal{T} \times \mathcal{T}$ is given as
\[ K(\rho, \tau) = \frac{1}{|\omega_{\tau}| |\omega_{\rho}|} |\omega_{\rho} \cap \omega_{\tau}|. \]

\subsection{Averaging kernels generated in reproducing kernel Hilbert spaces} \label{sec:AKHSbyRKHS}

As an important construction strategy for averaging kernels, we consider as initial Hilbert space $\mathcal{H}$ a RKHS $\mathcal{N}_{\Phi}(\Omega)$ generated by a symmetric positive definite reproducing kernel $\Phi \in C(\Omega \times \Omega)$ on a domain $\Omega \subseteq \mathbb{R}^d$. This strategy will allow us to obtain the averaging kernels for a large class of Hilbert spaces. Since RKHS's are typically linked to smoother function spaces, also the AKHS's build on these spaces will have a respectively larger smoothness. 

Consider a collection $\{\omega_{\tau} \ : \ \tau \in \mathcal{T}\}$ of subdomains in $\Omega$. As we assume that the reproducing kernel $\Phi$ is continuous, and since the point evaluation functionals in $\mathcal{N}_{\Phi}(\Omega)$ are continuous, by the positive definiteness of $ \Phi$ we immediately get, for any $f \in \mathcal{N}_{\Phi}(\Omega)$, the bound
\[ \left| \frac{1}{|\omega_\tau|}\int_{\omega_{\tau}} f(x) \mathrm{d} x \right| \leq \frac{1}{|\omega_\tau|}\int_{\omega_{\tau}} |\langle f, \Phi(\cdot, x) \rangle_{\mathcal{N}_{\Phi}(\Omega)}| \mathrm{d} x
\leq \frac{\| f \|_{\mathcal{N}_{\Phi}(\Omega)}}{|\omega_\tau|}\int_{\omega_{\tau}} \sqrt{\Phi(x, x)} \mathrm{d} x \leq \|\Phi\|_{\infty}^{\frac{1}{2}} \| f \|_{\mathcal{N}_{\Phi}(\Omega)}. \]
This implies that all averaging functionals $\lambda_{\tau}$ in $\mathcal{N}_{\Phi}(\Omega)$ are well-defined and continuous. We can therefore introduce an averaging kernel $A(x,\tau)$, and get 
\[ \langle f, A(\cdot, \tau) \rangle_{\mathcal{N}_{\Phi}(\Omega)} = \frac{1}{|\omega_\tau|}\int_{\omega_{\tau}} f(x) \mathrm{d} x = \frac{1}{|\omega_\tau|}\int_{\omega_{\tau}} \langle f, \Phi(\cdot, x) \rangle_{\mathcal{N}_{\Phi}(\Omega)} \mathrm{d} x = \left\langle f, \frac{1}{|\omega_\tau|}\int_{\omega_{\tau}} \Phi(\cdot, x) \mathrm{d} x \right\rangle_{\mathcal{N}_{\Phi}(\Omega)} \]
for all $f \in \mathcal{N}_{\Phi}(\Omega)$. Therefore, the averaging kernel has the form
\begin{equation} \label{eq:averagingkernel}
A(x,\tau) = \frac{1}{|\omega_\tau|}\int_{\omega_{\tau}} \Phi(x, y) \mathrm{d} y,
\end{equation} 
and it can be calculated in terms of mean values of the reproducing kernel $\Phi$. Furthermore, as a consequence of Fubini's theorem, the reproducing kernel $K(\rho,\tau)$ can also be formulated as
\begin{equation} \label{eq:reproducingkernel}
K(\rho,\tau) = \frac{1}{|\omega_\rho|} \int_{\omega_{\rho}} A(x,\tau) \de x = \frac{1}{|\omega_\rho|} \frac{1}{|\omega_\tau|} \int_{\omega_{\rho}} \int_{\omega_{\tau}} \Phi(x, y) \, \mathrm{d} y \mathrm{d} x.
\end{equation}

\subsection{Uniform averaging kernels based on translates of a single domain}
As a further more structured example of an AKHS, we consider the spaces $L_2(\mathbb{R}^d)$ with the averaging functionals given by the translates of a single subset $\omega \subseteq \mathbb{R}^d$. For further simplicity, we suppose that $\omega $ is a symmetric set centered at the origin, i.e., with $x \in \omega$ also $-x \in \omega$. 
Indexed by the set $\mathcal{T} = \mathbb{R}^d$, we then consider the subdomains $\{\omega_y = y + \omega \ : \ y \in \mathbb{R}^d\}$. Based on our assumptions, the averaging operations can now be formulated as
\begin{equation} \label{eq:convolutionindicatorfunction}
  \lambda_y(f) = \frac{1}{|\omega|} \int_{y + \omega} f(y) \mathrm{d} y =  \left( f \ast \frac{1}{|\omega|} \chi_{\omega} \right) (y).  
\end{equation}
In particular, the averaging functionals correspond to the point evaluation of the convolution of the function $f$ with the normalized indicator function $\frac{1}{|\omega|} \chi_{\omega}$. Several averaging kernels can be constructed explicitly in this framework.

\subsubsection{B-splines in the univariate case} 
As a first simple example in the univariate case, we start with the Hilbert space $L_2(\mathbb{R})$ and the domain $\omega = [-a/2,a/2]$ with length $a > 0$. To obtain the averaging operation in \eqref{eq:convolutionindicatorfunction}, the corresponding averaging kernel in $L_2(\mathbb{R})$ is given by the indicator function
\begin{equation} \label{eq:indicatorkernel1D}
A_1(x,y) = \frac{1}{a}\chi_{[y-a/2,y+a/2]}(x) = \left\{ \begin{array}{ll} 1/a & \text{if $|x-y| \leq a/2$} \\ 0 & \text{if $|x-y| > a/2$} \end{array} \right. .    
\end{equation}
Moreover, as respective associated reproducing kernel we get
\[K_1(x,y) = \frac{1}{a}\left( 1 - \frac{|x-y|}{a} \right)_+, \]
where $(x-y)_+ = \max(0,x-y)$ is known as positive part, or RELU function in the machine learning community. The kernels $A_1(x,y)$ and $K_1(x,y)$ can be written in terms of the central B-splines $M_1(x)$ and $M_2(x)$ as
\[A_1(x,y) = \frac{1}{a} M_1\left(\frac{x-y}{a}\right), \quad K_1(x,y) = \frac{1}{a} M_2\left(\frac{x-y}{a}\right).\] 

We can extend this using the $n$-th central B-spline $M_n(x)$. It is generally defined as the $(n-1)$-fold convolution of the indicator function $M_1(x) = \chi_{[-1/2,1/2]}(x)$, such that
\[
M_n(x) = \underbrace{(M_1 \ast M_1 \ast \cdots \ast M_1)}_{\text{($n-1$ convolutions)}} (x),
\]
which has the explicit representations \citep[see][]{butzer1988}
\[
M_n(x) = \frac{1}{(n-1)!}  \sum_{k=0}^{n} (-1)^k \binom{n}{k} \left(x +\frac{n}{2} - k\right)^{n-1}_+
\]
and
\begin{equation} \label{eq:Bsplines3}
M_n(x) = 
\begin{cases}
\displaystyle \frac{1}{(n-1)!}  \sum_{k=0}^{\lfloor n/2 - |x| \rfloor} (-1)^k \binom{n}{k} \left(\frac{n}{2} - |x| - k\right)^{n-1}, & x \in [-\frac{n}{2}, \frac{n}{2}], \\
0, & \text{otherwise}.
\end{cases}
\end{equation}
Here, $ \lfloor x \rfloor $ denotes the integer part of $ x $. 
By iteratively applying the convolution with the indicator function $\frac{1}{|\omega|} \chi_{\omega}$ to the kernel $K_1(x,y)$, we can 
generate further AKHS's with the averaging kernels
\[A_n(x,y) = \frac{1}{a} M_{2n-1}\left(\frac{x-y}{a} \right), \quad n \in \mathbb{N}.\] 
The corresponding associate reproducing kernels $K_n(x,y)$ are then given as 
\[K_n(x,y) = \frac{1}{a} M_{2n}\left(\frac{x-y}{a}\right), \quad n \in \mathbb{N}.\]
The positive definiteness of the reproducing kernels $ K_n (x,y) $ can be established directly. In fact, the cosine transform applied to \eqref{eq:Bsplines3} yields
\begin{equation} \label{eq:Bsplines4}
M_n(x) = \frac{1}{\pi} \int_0^\infty \left(\frac{\sin y/2}{y/2}\right)^n \cos(yx) \, \mathrm{d} y,
\end{equation}
and because of this, the Fourier transform of $M_{2n}(x)$ is non-negative and zero only for $x \in 2 \pi \mathbb{Z} \setminus \{0\}$. The reproducing kernels $K_n(x,y)$ are therefore positive definite and generate a native space $\mathcal{N}_{K_n}(\mathbb{R})$ that is a subspace of the Sobolev space $H^{2n+1}(\mathbb{R})$, as the B-splines $M_{2n}$ are elements of $H^{2n+1}(\mathbb{R})$. We will return on the role of the Fourier transform in the context of AKHS's in Section \ref{sec:Fourier}.

\begin{table}[]
    \centering
\begin{adjustbox}{width=1\textwidth}
\begin{tabular}{|c|c|c|}
\hline
$n$ & $A_n(x,y)$ & $K_n(x,y)$  \\
\hline 
1 & $\frac{1}{a}\chi_{[-a/2,a/2]}(x-y)$ & $\frac{1}{a}\left( 1 - \frac{|x-y|}{a} \right)_+$ \\[2mm]
2 & $\left\{ \begin{array}{ll}
     \frac{1}{a} \left(\frac{3}{2} - \frac{|x-y|^2}{a^2} \right) &  \text{if }\frac{|x-y|}{a} \leq \frac12 \\
    \frac{1}{2 a} \left( \frac{3}{2} - \frac{|x-y|}{a}\right)^2 & \text{if }\frac12 < \frac{|x-y|}{a} \leq \frac{3}{2}  \\
     0 & \text{otherwise}
\end{array}\right.$ & $\left\{ \begin{array}{ll}
     \frac{1}{a} \left(\frac{2}{3} - \frac{|x-y|^2}{a^2} + \frac{1}{2}\frac{|x-y|^3}{a^3} \right) &  \text{if }\frac{|x-y|}{a} \leq 1 \\
    \frac{1}{6 a} \left( 2 - \frac{|x-y|}{a}\right)^3 & \text{if }1 < \frac{|x-y|}{a} \leq 2  \\
     0 & \text{otherwise}
\end{array}\right.$ \\
\hline
\end{tabular}
\end{adjustbox}
    \caption{Simple examples of B-spline averaging and reproducing kernels}
    \label{tab:Bsplines}
\end{table}

\subsubsection{Radial uniform averaging kernels for $d = 1$}

A uniform averaging kernel $A: \mathbb{R}^d \times \mathbb{R}^d \to \mathbb{R} $ is called \emph{radial} if 
\[A(x,y) = \alpha(\|x-y\|_2)\]
for some univariate function $\alpha: [0, \infty) \to \mathbb{R}$. 
For a uniform AKHS with a radial kernel $A(x,y)$ and a (radial) ball $\omega$ as an underlying averaging domain, also the associated reproducing kernel $K(x,y)$ is radial and can be written as
\[ K(x,y) = \kappa(\|x-y\|_2)\]
with a univariate function $\kappa: [0, \infty) \to \mathbb{R}$. 

In the univariate setting $d = 1$, and using the domain $\omega = [-a/2,a/2]$ with interval length $a$, the respective averaging kernel can be written as $A(x,y) = \alpha(x-y)$ by extending $\alpha$ evenly to the entire real axis. The same holds true for the reproducing kernel $K(x,y) = \kappa(x-y)$, that can now be expressed in terms of the even function
\[ \kappa(x) = \left(\alpha \ast \frac{1}{a} \chi_{[-a/2,a/2]}\right)(x). \]
As the reproducing kernel $K$ is positive semi-definite, the function $\kappa$ itself is a positive semi-definite function. If the underlying space $\mathcal{H}$ is a native space $\mathcal{N}_{\Phi}(\Omega)$ based on a univariate radial kernel $\Phi$, such that
\[\Phi(x,y) = \phi(x-y), \]
the uniform averaging kernel $A$ and the reproducing kernel $K$ can be calculated in terms of $\phi$ as
\[ \alpha(x) = \phi \ast \frac{1}{a}\chi_{[-a/2,a/2]}(x), \quad \kappa(x) = \phi \ast \frac{1}{a} \chi_{[-a/2,a/2]} \ast \frac{1}{a} \chi_{[-a/2,a/2]}(x).\]
In particular, if the function $\phi$ is positive definite, also $\kappa$ is positive definite. We encountered already the central B-splines as examples of such kernels. If the first two anti-derivatives of $\phi$ exist, the functions $\alpha$ and $\kappa$ can be stated explicitly by the fundamental theorem of calculus.

\begin{figure}[htbp]
    \centering
    \includegraphics[width=0.49\linewidth]{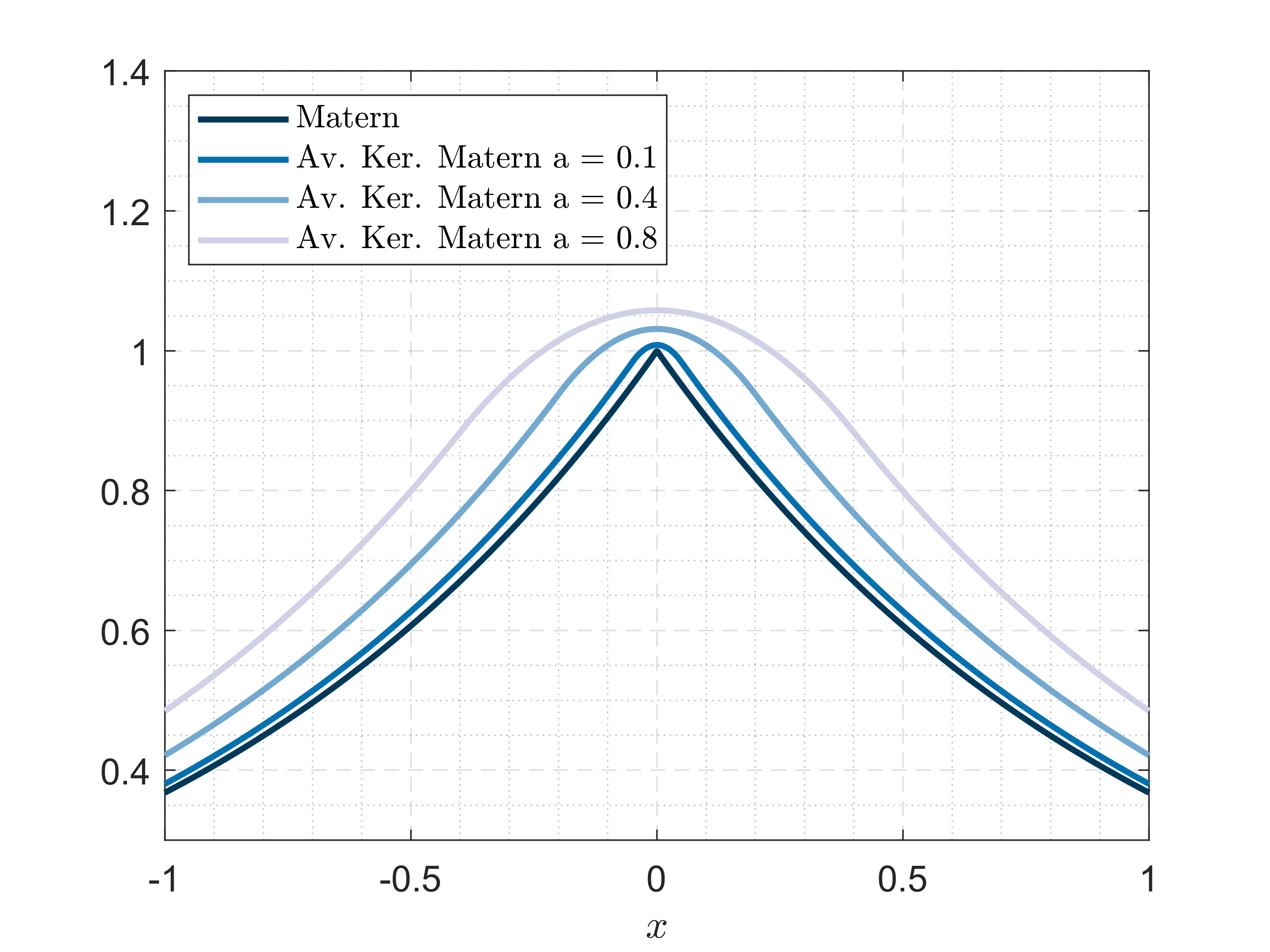}
    \includegraphics[width=0.49\linewidth]{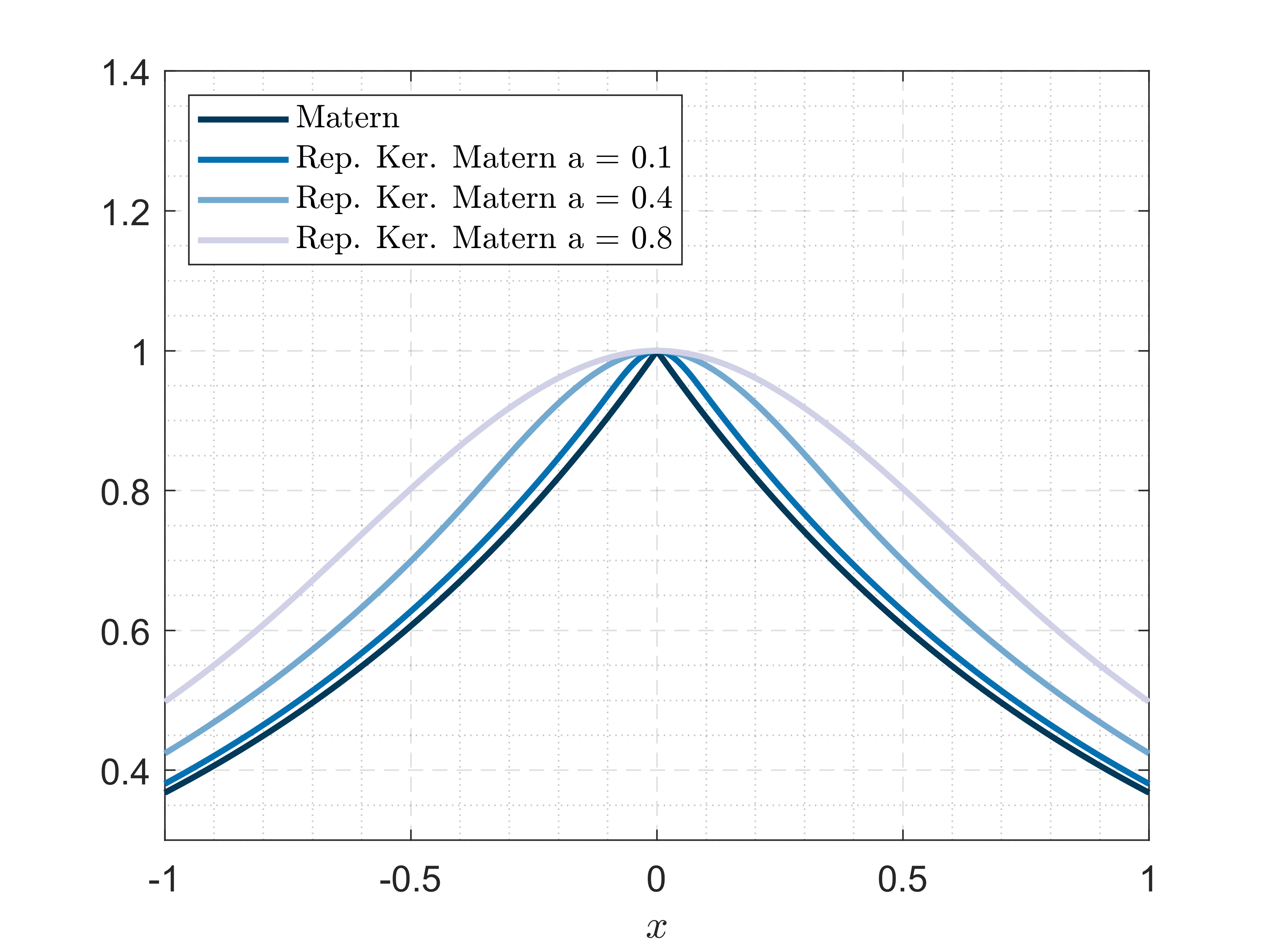}
    \caption{The averaging and reproducing kernel based on the Mat\'ern function. Left: the Mat\'ern function $\phi(x) = e^{-\lambda |x|}$ compared to the averaging kernels $\alpha(x)/\kappa(0)$ given in Example \ref{ex:RUAK} (i) for three interval lengths $a \in \{0.1, 0.4, 0.8\}$. Right: the respective normalized reproducing kernels $\kappa(x)/\kappa(0)$.}
    \label{fig:maternkernel}
\end{figure}

\begin{lemma} \label{lem:ftc}
Let $I_\phi(x)$ and $I_\phi^2(x)$ denote the first and second anti-derivative of the univariate function $\phi$ that generates the native space $\mathcal{N}_{\Phi}(\Omega)$. Then, the radial functions $\alpha$ and $\kappa$ of the averaging kernel $A$ and the reproducing kernel $K$ can be written as
\[\alpha(x) = \frac{1}{a} \left(I_{\phi}(x+a/2) - I_{\phi}(x-a/2)\right), \quad \kappa(x) = \frac{1}{a^2} \left(I_{\phi}^2(x+a) + I_{\phi}^2(x-a) - 2 I_{\phi}^2(x) \right),\]
i.e., $\alpha$ and $\kappa$ are first- and second-order symmetric finite differences of the first and second anti-derivative of $\phi$ with respect to the step size $a$, respectively. 
\end{lemma}

\noindent Lemma \ref{lem:ftc} gives a convenient way to compute averaging kernels explicitly, as the next example shows.

\begin{example} \label{ex:RUAK}
\begin{itemize}
    \item[$(i)$] For the Mat\'ern function $\phi(x) = e^{-\lambda |x|}$, we get
    \[\alpha(x) = \left\{ \begin{array}{ll}
\frac{1}{\lambda a} (2 - e^{\lambda (|x|-a/2)} - e^{-\lambda (|x|+a/2)}) & \text{if} \; |x| \leq a/2, \\
\frac{1}{\lambda a} (e^{-\lambda (|x|-a/2)} - e^{-\lambda (|x|+a/2)}) & \text{if} \; |x| > a/2,
\end{array} \right. \]
and 
    \[\kappa(x) = \left\{ \begin{array}{ll}
\frac{1}{\lambda^2 a^2} (2 \lambda (a - |x|) + e^{\lambda (|x|-a)} + e^{-\lambda (|x|+a)} - 2 e^{-\lambda (|x|}) & \text{if} \; |x| \leq a, \\
\frac{1}{\lambda^2 a^2} (e^{-\lambda (|x|-a)} + e^{-\lambda (|x|+a) } - 2 e^{-\lambda (|x|)}) & \text{if} \; |x| > a.
\end{array} \right. \]
\item[$(ii)$] In case of the inverse quadratic $\phi(x) = \frac{1}{1+ (\lambda x)^2}$ we get 
\begin{align*}
\alpha(x) &= \frac{1}{\lambda a} \left( \arctan (\lambda (x +a/2)) - \arctan (\lambda (x - a/2)) \right) \\ &= \frac{1}{\lambda a}\arctan\left(\frac{\lambda a}{1 + \lambda^2 \left( x^2 - \frac{a^2}{4} \right)}\right),
\end{align*}
where the last identity holds if $\lambda^2 (\frac{a}{2})^2 < 1$. For the function $\kappa$ of the radial kernel $K$ we get
\begin{align*}
\kappa(x) &= \frac{1}{\lambda^2 a^2} \left( \lambda(x+a)\arctan (\lambda (x +a)) + \lambda(x-a) \arctan (\lambda (x - a)) - 2 \lambda x \arctan (\lambda x ) \right)  \\
& \quad - \frac{1}{\lambda^2 a^2} \left( \frac12 \ln (1 + \lambda (x +a)) + \frac12 \ln (1 + \lambda (x - a )) - \ln (1 + \lambda x) \right) \\
&= \frac{1}{\lambda a}   \arctan\left(\frac{2\lambda a}{1 + \lambda^2 (x^2 - a^2)}\right) -
 \frac{x}{\lambda a^2} \left( \arctan\left(\frac{2\lambda^3 a^2 x}{(1 + \lambda^2 x^2)^2 + \lambda^2 a^2 (1 - \lambda^2 x^2)}\right)  \right) 
 \\ & \quad - \frac{1}{2\lambda^2 a^2} \ln \left( 1 - \frac{\lambda^2 a^2}{(1 + \lambda x)^2} \right),
\end{align*}
where the last identity is true if $\lambda^2 a^2 < 1$. 
\item[$(iii)$] For the inverse multiquadric $\phi(x) = \frac{1}{\sqrt{1+ (\lambda x)^2}}$ we get 
\[\alpha(x) = \frac{1}{\lambda a} \left( \sinh^{-1} (\lambda (x +a/2)) - \sinh^{-1} (\lambda (x - a/2)) \right)\]

and 
\begin{align*}
\kappa(x) &= \frac{1}{\lambda^2 a^2} \left( \lambda(x+a)\sinh^{-1} (\lambda (x +a)) + \lambda(x-a) \sinh^{-1} (\lambda (x - a)) - 2 \lambda x \sinh^{-1} (\lambda x ) \right)  \\
& \quad - \frac{1}{\lambda^2 a^2} \left( \sqrt{1 + \lambda^2 (x + a)^2} + \sqrt{1 + \lambda^2 (x - a )^2} - 2\sqrt{1 + \lambda^2 x^2} \right).
\end{align*}
\item[$(iv)$] Starting with the anti-derivatives $I_{\phi}^2$ and $I_{\phi}$, Lemma \ref{lem:ftc} can be applied also in a slightly different way to obtain averaging and reproducing kernels explicitly. As an example, it is well-known that $I_{\phi}^2(x) = - \frac{1}{2 \lambda }e^{- \lambda x^2}$ with $\lambda > 0$ is the first and second anti-derivative of the functions
\[I_{\phi}(x) = x e^{- {\lambda} {x^2}}, \quad\text{and} \quad \phi(x) = \left( 1 - 2 { \lambda} x^2 \right) e^{- { \lambda}{x^2}}.\]
The function $\phi(x)$ is positive definite and known as Mexican hat wavelet. Using Lemma \ref{lem:ftc}, we immediately get the averaging kernel for $\phi$ as
\[\alpha(x) = \left( \frac{1}{2} + \frac{x}{a} \right)  e^{- {\lambda} {(x+\frac{a}{2})^2}} + \left( \frac{1}{2} - \frac{x}{a} \right) e^{- {\lambda} {(x - \frac{a}{2})^2}}\]
and the respective reproducing kernel as
\[ \kappa(x) = \frac{1}{2 \lambda a^2} \left(2 e^{- {\lambda}{x^2}} - e^{- {\lambda} (x+a)^2} - e^{- {\lambda}{(x-a)^2}} \right).\]
\end{itemize}
\end{example}

\subsubsection{Radial uniform averaging kernels in higher dimensions $d > 1$}

Already in the univariate case $ d = 1 $, Lemma \ref{lem:ftc} indicates that not all averaging kernels can be expressed through elementary functions even when the kernel $\Phi$ itself can. A typical example is the Gaussian kernel $ e^{-\lambda | x|^2} $, whose anti-derivative involves the error function.

For $ d > 1 $ already the simple functions considered in Example \ref{ex:RUAK} do not lead to explicit formulas for $A$ and $K$. Using special functions, we can nevertheless derive the following result.

\begin{lemma} Let $\Phi: \mathbb{R}^d \times \mathbb{R}^d \to \mathbb{R}$ be a radial kernel based on the univariate function $\phi$. Then, the uniform radial averaging kernel $A$ based on the ball $\omega = \{x \in \mathbb{R}^d \, : \, \|x\|_2 \leq a\}$ with radius $a$ and centered at the origin is given as $A(x,y) = \alpha(\|x-y\|_2)$ with
\begin{equation}
\label{eq:uniformgeneral}
\alpha(r)
= \left\{ \begin{array}{ll}\frac{d}{2 a^d}
\displaystyle \int_{r-a}^{\,r+a} \textstyle
\phi(\rho)\,\rho^{d-1}\,
I_{1- \left(\frac{\rho^2 + r^2 - a^2}{2\rho r}\right)^2}\left(\frac{d-1}{2},\,\frac{1}{2}\right)
\,\mathrm{d}\rho & \text{if $r \geq a $,} \\
\frac{d}{a^d} \displaystyle \int_{0}^{a-r} \textstyle
\phi(\rho)\,\rho^{d-1}\,\mathrm{d}\rho \\[2mm] 
 \quad + \frac{d}{ a^d}
\displaystyle\int_{a-r}^{\sqrt{(a-r)(a+r)}} \textstyle
\phi(\rho)\,\rho^{d-1}\left(1 - \frac{1}{2}
I_{1- \left(\frac{\rho^2 + r^2 - a^2}{2\rho r}\right)^2}\left(\frac{d-1}{2},\,\frac{1}{2}\right) \right)
\,\mathrm{d}\rho \\
 \quad + \frac{d}{2 a^d}
\displaystyle \int_{\sqrt{(a-r)(a+r)}}^{a+r} \textstyle 
\phi(\rho)\,\rho^{d-1}\,
I_{1- \left(\frac{\rho^2 + r^2 - a^2}{2\rho r}\right)^2}\left(\frac{d-1}{2},\,\frac{1}{2}\right) 
\,\mathrm{d}\rho & \text{if $0 \leq r < a $,}
\end{array}\right.
\end{equation}
where $I_z(a,b)$ denotes the regularized incomplete beta function given as
\[I_z(a,b) = \frac{\Gamma(a+b)}{\Gamma(a) \Gamma(b)}\int_0^z t^{a-1}(1-t)^{b-1} \mathrm{d} t,\]
and $\Gamma$ the well-known Gamma function. For the reproducing kernel $K$ we get analogous formulas by replacing in formula \eqref{eq:uniformgeneral} the function $\phi$ by the function $\alpha$. 
\end{lemma}

\begin{proof}
By the construction of the averaging kernel $A$ in \eqref{eq:averagingkernel} and the fact that $\chi_{\omega}$ and $\phi(\|x\|_2)$ are radial, we see that also the averaging kernel $A$ is radial and get
\begin{equation*} 
\alpha(r) = \frac{1}{|\omega|}\int_{\|x - z\|_2 \leq a} \phi(\|x\|_2)\,\mathrm{d}x,    
\end{equation*} 
for any $z \in \mathbb{R}^d$ with $\|z\|_2 = r \geq 0$. Using spherical coordinates centered at the origin, we can express $x \in \mathbb{R}^d$ as $x = \rho \nu$ with $\rho = \|x\|_2$ and $\nu = x/\|x\|_2$. The corresponding volume element is $\mathrm{d}x = \rho^{d-1} \mathrm{d} \nu \mathrm{d} \rho$, transforming the integral above into
\begin{equation} \label{eq:cap1}\alpha(r) = \frac{1}{|\omega|} \int_0^\infty \phi(\rho) \rho^{d-1}  \left( \int_{\|\rho \nu - z\|_2 \leq a} \,\mathrm{d} \nu \right) \, \mathrm{d} \rho. 
\end{equation}
The set 
\begin{equation} \label{eq:sphericalcap}
\{ \nu \in \mathbb{S}^{d-1} \ : \ \|\rho \nu - z\|_2 \leq a  \} = \left\{ \nu \in \mathbb{S}^{d-1} \ : \ \frac{\rho^2 + r^2 - a^2}{2 \rho r} \leq \nu^ T \frac{z}{\|z\|_2} \right\}    
\end{equation}
is either empty, the unit sphere $\mathbb{S}^{d-1}$, or a spherical cap of the unit sphere $\mathbb{S}^{d-1}$ with center $z/r$ and polar angle $\theta \in [0,\pi]$ defined by $\cos \theta = \frac{\rho^2 + r^2 - a^2}{2 \rho r}$ (if the latter term lies in $[-1,1]$). 
In particular, according to the definition and the non-negativity $\rho \geq 0$, the set \eqref{eq:sphericalcap} is non-empty if and only if $\max\{0,r-a\} \leq \rho \leq r + a$. We have to distinguish two cases:

\vspace{2mm}
\noindent $(i)$ $r \geq a$: in this case, we have $0 \leq \frac{\rho^2 + r^2 - a^2}{2 \rho r} \leq 1$ for $r-a \leq \rho \leq r + a$. This means that the polar angle $\theta \in [0,\frac{\pi}{2}]$, and according to \citep{Li2011}, the area of the spherical cap is given as
\[ \frac{1}{2} \, |\mathbb{S}^{d-1}|\,I_{1-\cos^2 \theta}\left(\frac{d-1}{2},\,\frac{1}{2}\right).
\]
This area corresponds to the inner integral in \eqref{eq:cap1}. With $|\omega|$ being the volume of the unit ball with radius $a$, we therefore get 
\[\alpha(r) = \frac{|\mathbb{S}^{d-1}|}{2 |\omega|} \int_{r-a}^{r+a} \phi(\rho) \rho^{d-1}  I_{1-\cos^2 \theta}\left(\frac{d-1}{2},\,\frac{1}{2}\right) \, \mathrm{d} \rho \]
with 
\begin{equation} \label{eq:ratiosphere}
 \frac{|\mathbb{S}^{d-1}|}{|\omega|} = \frac{2 \pi^{\frac{d}{2}}}{\Gamma(\frac{d}{2})} \frac{\Gamma(\frac{d}{2}+1)}{\pi^{\frac{d}{2} } a^d} = \frac{d}{a^d}.   
\end{equation}
\noindent $(ii)$ $r < a$: in this case, we have to subdivide the integral in three parts. For $\rho \in [0,a-r]$, the set \eqref{eq:sphericalcap} corresponds to the entire unit sphere $\mathbb{S}^{d-1}$, providing in combination with \eqref{eq:ratiosphere} the first integral term in the composed formula \eqref{eq:uniformgeneral}. \\ For $\rho \in [\sqrt{(a-r)(a+r)},a+r]$, the polar angle $\theta$ of the spherical cap is in $[0,\frac{\pi}{2}]$. For this range of $\rho$, we therefore get the same integrand as in case $(i)$, and, thus, the last integral term in \eqref{eq:uniformgeneral}. \\
For the remaining interval $\rho \in [a-r, \sqrt{(a-r)(a+r)}]$, the polar angle of the spherical cap is in $[\frac{\pi}{2}, \pi]$. By symmetry, the area of the spherical cap is in this case given as
\[ |\mathbb{S}^{d-1}| \left(1 - \frac{1}{2}\,I_{1-\cos^2 \theta}\left(\frac{d-1}{2},\,\frac{1}{2}\right) \right),
\]
giving the second integral term in the formula \eqref{eq:uniformgeneral}.
\end{proof}

\begin{example} \label{ex:nDkernels}
Whereas it is typically not possible to get explicit expressions for the integrals in \eqref{eq:uniformgeneral}, at least the value 
\[\alpha(0) = \frac{d}{a^d} \displaystyle \int_{0}^{a} \textstyle
\phi(\rho)\,\rho^{d-1}\,\mathrm{d}\rho\]
for $r = 0$ can be calculated directly in many cases: 
    \begin{itemize}
        \item[(i)] For the multivariate Mat\'ern kernel $\Phi(x,y) = e^{-\lambda \Vert x - y \Vert_2}$, $\lambda > 0$, we get 
         $$ \alpha(0) = \frac{d}{a^d} \displaystyle \int_{0}^{a} \textstyle
e^{-\lambda \rho}\,\rho^{d-1}\,\mathrm{d}\rho  = d \frac{\gamma (d,\lambda a)}{(\lambda a) ^d } ,$$ where $\gamma $ denotes the lower incomplete gamma function.
         \item[(ii)] For the multivariate Gauss kernel $\Phi(x,y) = e^{-\lambda \Vert x - y \Vert_2^2 }$, $\lambda > 0$, we get that 
    $$\alpha(0) = \frac{d}{a^d} \displaystyle \int_{0}^{a} \textstyle
e^{-\lambda \rho^2}\,\rho^{d-1}\,\mathrm{d}\rho  = \frac{d}{2} \frac{\gamma (\frac{d}{2},\lambda a^2)}{(\lambda a^2) ^\frac{d}{2} }. $$
    \end{itemize}
    \end{example}

In a general scenario, quadrature rules are necessary to approximate the involved integrals. We will discuss this in Section \ref{sect:implementationstrategy}.

\section{Characterization of AKHS in terms of the Fourier transform} \label{sec:Fourier}

RKHS's associated with positive semi-definite kernels defined on the entire Euclidean space $\mathbb{R}^d$ can be characterized in terms of their (generalized) Fourier transform. Similar characterizations are also possible for AKHS's with radial uniform averaging kernels. Since the Fourier transform of discontinuous functions is, in general, not integrable, we will technically employ the Fourier-Plancherel transform to characterize these spaces. We begin with the univariate case.

\begin{theorem} \label{thm:charAKHS}
An even non-vanishing function $\alpha \in L_{1}(\mathbb{R}) \cap L_2(\mathbb{R})$ generates a uniform averaging kernel $A(x,y)$ based on the domain $\omega = [-a/2,a/2]$ if and only if $\alpha$ is the inverse Fourier-Plancherel transform of an even function $\hat{\alpha} \in C_0(\mathbb{R}) \cap L_2(\mathbb{R})$ with the property that
\begin{equation} \label{eq:charAKHS}
\hat{\alpha}(s) \left\{ \begin{array}{ll}
\geq 0 & \text{if} \; |s| \in [ (2k) \frac{2\pi}{a}, (2k + 1)\frac{2\pi}{a}], \quad k \in \mathbb{N}_0, \\
\leq 0 & \text{if} \; |s| \in [(2 k + 1) \frac{2\pi}{a}, (2 k + 2)  \frac{2\pi}{a}], \quad k \in \mathbb{N}_0
\end{array} \right. .    
\end{equation}
\end{theorem}

\begin{proof}
For an even function $\alpha \in L_{1}(\mathbb{R}) \cap L_2(\mathbb{R})$ its Fourier transform 
\[\hat{\alpha}(s) = \frac{1}{\sqrt{2 \pi}} \int_{\mathbb{R}} \alpha(x) e^{-i s x} \mathrm{d} x\]
is continuous and square-integrable. The associated reproducing kernel $\kappa(x) = (\alpha \ast \frac{1}{a} \chi_{[-a/2,a/2]})(x)$ is also an element of $L_{1}(\mathbb{R}) \cap L_2(\mathbb{R})$ and, according to the convolution theorem, its Fourier transform is given as
\[\hat{\kappa}(s) = \frac{\sqrt{2 \pi}}{a} \hat{\alpha}(s) \widehat{\chi}_{[-a/2,a/2]}(s) =  \hat{\alpha}(s) \frac{2}{a s} \sin\left( \frac{a s}{2} \right).\]
As $\kappa$ gives rise to a positive semi-definite kernel, the Fourier coefficients $\hat{\kappa} \geq 0$ are non-negative. Because of the periodicity of the sine function, the Fourier coefficients $\hat{\alpha}(s)$ therefore have to satisfy the oscillating property \eqref{eq:charAKHS}.

On the other hand, if the Fourier coefficients of $\alpha \in L_{1}(\mathbb{R}) \cap L_2(\mathbb{R})$ are oscillating as in \eqref{eq:charAKHS}, the Fourier coefficients $\hat{\kappa}(s) = \hat{\alpha}(s) \frac{2}{a s} \sin\left( \frac{a s}{2} \right)$ of the function $\kappa = \alpha \ast \frac{1}{a} \chi_{[-a/2,a/2]}$ are non-negative. Further,
\[ \|\hat{\kappa}\|_1 \leq \|\hat{\alpha}\|_2 \left\|\frac{2}{a s} \sin\left( \frac{a s}{2} \right)\right\|_2 =  \frac{\sqrt{2 \pi}}{\sqrt{a}}\|\alpha\|_2.\]
Therefore, by Bochner's theorem, the function $\kappa$ is continuous and positive semi-definite and gives rise to a positive semi-definite kernel $K(x,y) = \kappa(x-y)$. By the Moore-Aronszajn theorem, there exists a unique RKHS $\mathcal{N}_{K}(\mathbb{R})$ with reproducing kernel $K$. 

The convolution operator $\Lambda \alpha = \alpha \ast \frac{1}{a} \chi_{[-a/2,a/2]}$ is an injective linear operator on $L_2(\mathbb{R})$. Therefore, $\Lambda$ defines a bijective mapping between the subspaces $\mathrm{span} \{\alpha(\cdot - y) \ : \ y \in \mathbb{R}\}$ and $\mathrm{span} \{\kappa(\cdot - y) \ : \ y \in \mathbb{R}\} \subseteq \mathcal{N}_{K}(\mathbb{R})$. By inducing the norm of $\mathcal{N}_{K}(\mathbb{R})$ on $\mathrm{span} \{\alpha(\cdot - y) \ : \ y \in \mathbb{R}\}$ such that
\[\left\|\sum_{i=1}^n \beta_i \alpha(\cdot- y_i)\right\|^2_{\mathcal{N}_{A}(\mathbb{R})} := \left\|\sum_{i=1}^n \beta_i \kappa(\cdot- y_i)\right\|^2_{\mathcal{N}_K(\mathbb{R})},\]
the space $\mathrm{span} \{\alpha(\cdot - y) \ : \ y \in \mathbb{R}\}$ turns into a pre-Hilbert space which is isometric to the pre-Hilbert space $\mathrm{span} \{\kappa(\cdot - y) \ : \ y \in \mathbb{R}\}$ with the norm of $\mathcal{N}_K(\mathbb{R})$. Going over to the closure of the two spaces, we then get an isometric isomorphism $\Lambda$ between two Hilbert spaces $\mathcal{N}_A(\mathbb{R})$ and $\mathcal{N}_K(\mathbb{R})$. Even more, the kernel $A(x,y) = \alpha(x-y)$ is an averaging kernel since
\[ \langle f, A(\cdot, y)\rangle_{\mathcal{N}_A(\mathbb{R})} = \langle \Lambda f, K(\cdot, y)\rangle_{\mathcal{N}_K(\mathbb{R})} = \Lambda f(y) \]
holds true for the RKHS $\mathcal{N}_K(\mathbb{R})$. This makes $\mathcal{N}_A(\mathbb{R})$ to an AKHS with averaging kernel $A$. The fact that the RKHS $\mathcal{N}_K(\mathbb{R})$ is unique and the convolution operator $\Lambda$ injective, makes also the AKHS $\mathcal{N}_A(\mathbb{R})$ unique. 
\end{proof}

More generally, for radial uniform kernels $A(x,y) = \alpha(\|x-y\|_2)$ in $\mathbb{R}^d$ we get: 

\begin{theorem} \label{thm:charAKHSmulti}
A univariate non-vanishing function $\alpha \in L_{1}([0,\infty),r^{d-1}) \cap L_2([0,\infty),r^{d-1})$ generates a multivariate uniform radial averaging kernel $A(x,y)$ in $\mathbb{R}^d$ based on the ball $\omega = \{x \in \mathbb{R}^d \ : \ \|x\|_2 \leq a\}$ if and only if $\alpha$ is the inverse Hankel transform of order $\frac{d}{2} - 1$ of a function $\hat{\alpha} \in C_0([0,\infty)) \cap L_2([0,\infty),r^{d-1})$ with the property that
\begin{equation} \label{eq:charAKHSradial}
\hat{\alpha}(s) \left\{ \begin{array}{ll}
\geq 0 & \text{if} \; s \in [ 0, j_{\frac{d}{2},1}(a)], \\
\leq 0 & \text{if} \; s \in [ j_{\frac{d}{2},2k-1}(a), j_{\frac{d}{2},2k}(a)], \quad k \in \mathbb{N}, \\
\geq 0 & \text{if} \; s \in [ j_{\frac{d}{2},2k}(a), j_{\frac{d}{2},2k+1}(a)], \quad k \in \mathbb{N}, 

\end{array} \right.    
\end{equation}
where $0 < j_{\frac{d}{2},1}(a) < j_{\frac{d}{2},2}(a) < \cdots$, denote the zeros of the scaled Bessel function $J_{\frac{d}{2}}(a s)$ of order $\frac{d}{2}$. 
\end{theorem}

\begin{proof} The proof follows pretty much the lines of the univariate case. For $\alpha \in L_{1}([0,\infty),r^{d-1}) \cap L_2([0,\infty),r^{d-1})$ the Hankel transform 
\[\hat{\alpha}(s) = s^{-\frac{d-2}{2}} \int_{0}^\infty \alpha(r) r^{\frac{d}{2}} J_{\frac{d-2}{2}}(s r) \mathrm{d}r\]
is continuous and in $L_2([0,\infty),r^{d-1})$. The reproducing kernel $K(x,y) = \kappa(\|x-y\|_2)$ is also radial and its generating function $\kappa(\|z\|_2)$ corresponds to the convolution
\[ \kappa(\|z\|_2) = \frac{1}{|\omega|}\int_{\|x - z\|_2 \leq a} \alpha(\|x\|)_2) \mathrm{d} x\]
of $\alpha(\|x\|)_2$ with the characteristic function of the ball $\omega$. For this, by the convolution theorem, the Hankel transform of $\kappa$ can be written as
\begin{equation} \label{eq:hankeltransformkernel}\hat{\kappa}(s) = \frac{(2 \pi )^{\frac{d}{2}}}{|\omega|} \hat{\alpha}(s) \hat{\chi}_\omega(s) = \frac{\Gamma(\frac{d}{2}+1)}{(\pi a s)^{\frac{d}{2}}} J_{\frac{d}{2}}(a s) \hat{\alpha}(s), \quad s \geq 0.
\end{equation}
As for an AKHS the reproducing kernel $K$ has to be positive semi-definite, the Fourier coefficients $\hat{\kappa}(s)$ have to be non-negative, implying that the sign of $\hat{\alpha}(s)$ is identical to the sign of $J_{\frac{d}{2}}(a s)$. This determines the condition in \eqref{eq:charAKHSradial}. Regarding the reverse statement, it follows by the same construction principle upon the RKHS $\mathcal{N}_K(\mathbb{R}^d)$ as shown in the univariate case in Theorem \ref{thm:charAKHS}. 
\end{proof}

\begin{remark} The reproducing kernel $K(x,y) = \kappa(\|x-y\|_2)$ associated to the averaging kernels $A(x,y)$ in Theorem \ref{thm:charAKHS} and Theorem \ref{thm:charAKHSmulti} is in fact positive definite.
Since the Hankel transform $\hat{\alpha}(s)$ is continuous and non-vanishing, also the function $\hat{\kappa}(s)$ in \eqref{eq:hankeltransformkernel} is continuous and non-vanishing. This implies that the support of $\hat{\kappa}$ contains an open subset, and, thus, by \citep[Theorem 6.8]{Wendland_2004} that the kernel $K(x,y)$ is positive definite in $\mathbb{R}^d$. 
\end{remark}

\begin{remark} As outlined in the proof of Theorem \ref{thm:charAKHS} (and in a similar fashion also for the multivariate result in Theorem \ref{thm:charAKHSmulti}), the radial uniform AKHS $\mathcal{N}_A(\mathbb{R}^d)$ is uniquely determined by the generating function $\alpha$ and the fact that the chosen domains are translates of a ball $\omega$ with radius $a$.

This allows to interpret also the space $L_2(\mathbb{R}^d)$ as an AKHS. For this, we can consider the kernel $A(x,y) = \frac{1}{|\omega|} \chi_{\omega}(x-y)$ as a radial uniform averaging kernel inside $L_2(\mathbb{R}^d)$ with respect to the translates of the ball $\omega$. As the translates $\{ A(\cdot,y) \ : \ y \in \mathbb{R}^d \}$ of the normalized indicator function generate a dense subset of $L_2(\mathbb{R}^d)$ (this follows by \citep[Proposition (4.71)]{Folland95} and the fact that $\hat{\chi}_{\omega}$ vanishes only on a set of Lebesgue measure $0$), the respective AKHS corresponds to $L_2(\mathbb{R}^d)$, with the kernel $A(x,y)$ uniquely given in terms of the normalized and translated indicator function. We can observe that although the generated Hilbert space is always $L_2(\mathbb{R}^d)$, the averaging kernel depends on the selected domains. However, as soon as the domains are fixed, at least in the radial setting we have shown that also the averaging kernels are uniquely determined.  
\end{remark}

\section{Solution of histopolation problems} \label{sec:histopolation}

The AKHS setting turns out to be the right theoretical framework to solve histopolation problems with a kernel method.    
To state and solve the problem, we assume that $\mathcal{N}_A(\Omega) \subseteq \mathcal{H}$ is an averaging kernel Hilbert space based on averages created on the subdomains $\{\omega_{\tau} \ : \ \tau \in \mathcal{T}\}$. 
We further assume that we are given the mean values 
\begin{equation} \label{eq:rhs}
\lambda_{\tau_i}(f) = \frac{1}{|\omega_{\tau_i}|}\int_{\omega_{\tau_i}} f(x) \mathrm{d}x 
\end{equation}
of the function $f \in \mathcal{N}_A(\Omega)$ on $n$ subdomains
$\{\omega_{\tau_1}, \ldots, \omega_{\tau_n}\}$ based on the indices $\mathcal{T}_n = \{\tau_1, \ldots, \tau_n\} \subseteq \mathcal{T}$. Our aimed-at approximant $s_f$ of the function $f$ should retain these mean values in the histopolation space
$$ \mathcal{N}_{A}(\Omega,\mathcal{T}_n) = \mathrm{span} \left\{ A(\cdot, \tau_j) \ ; \ j \in \{1, \ldots n\} \right\} \subseteq \mathcal{N}_{A}(\Omega).$$
More precisely, we want the approximant 
\begin{equation} \label{eq:approxAKHS}
s_f(x) = \sum_{j=1}^n c_j A(x,\tau_j) \in \mathcal{N}_A(\Omega,\mathcal{T}_n)
\end{equation}
to satisfy the $n$ histopolation conditions
\[ \lambda_{\tau_i}(s_f) = \frac{1}{|\omega_{\tau_i}|} \int_{\omega_{\tau_i}} s_f(x) \mathrm{d} x = \frac{1}{|\omega_{\tau_i}|}\int_{\omega_{\tau_i}} f(x) \mathrm{d} x = \lambda_{\tau_i}(f), \qquad i \in \{1, \ldots n\}. \]
These $n$ conditions can be formulated in terms of a linear system of equations 
\begin{equation} \label{eq:mainhistopolationsystem}
 \mathbf{K} \, \mathbf{c} = \boldsymbol{\lambda},   
\end{equation}
where the vector $\mathbf{c} = [c_1 \cdots c_n]^\intercal$ contains the expansion coefficients of the histopolant $s_f$, \[ \boldsymbol{\lambda} = [\lambda_{\tau_1}(f) \; \cdots \; \lambda_{\tau_n}(f)]^\intercal\] contains the histopolation data \eqref{eq:rhs} and $\mathbf{K}$ is the $n \times n$ positive semi-definite matrix with the entries 
\begin{equation} \label{eq:histopolationmatrix}
\mathbf{K}_{i,j} = \frac{1}{|\omega_{\tau_i}|} \int_{\omega_{\tau_i}} A(x,\tau_j) dx = K(\tau_i,\tau_j). 
\end{equation}
In an AKHS, the linear independence of the functionals $\{ \lambda_{\tau_1}, \ldots, \lambda_{\tau_n} \}$ is equivalent to the linear independence of the elements $\{ A(\cdot,\tau_1), \ldots, A(\cdot,\tau_n) \}$. 
The averaging property of the AKHS further implies that
\[ 
      \left\| \sum_{i=1}^n \alpha_i A(\cdot,\tau_i) \right\|_{\mathcal{H}}^2 = \left\langle \sum_{i=1}^n \alpha_i A(\cdot,\tau_i), \sum_{j=1}^m \alpha_j A(\cdot,\rho_j) \right\rangle_{\mathcal{H}}   = \sum_{i=1}^n  \sum_{j=1}^m \alpha_i \alpha_j K(\tau_i, \rho_j) ,
\]
implying that the matrix $\mathbf{K}$ is positive definite and, thus, invertible if the functionals $\{ \lambda_{\tau_1}, \ldots, \lambda_{\tau_n} \}$ are linearly independent. The developed theory on averaging kernel Hilbert spaces implies now the following result.

\begin{theorem} \label{thm:histopolationunique}
If $\mathcal{N}_A(\Omega)$ is an averaging kernel Hilbert space with linearly independent averaging functionals $\{ \lambda_{\tau_1}, \ldots, \lambda_{\tau_n} \}$, then the matrix $\mathbf{K}$ is positive definite, and the system \eqref{eq:mainhistopolationsystem} has a unique solution. Further, the histopolant $s_f \in \mathcal{N}_{A}(\mathcal{T}_n)$ in \eqref{eq:approxAKHS} can be calculated explicitly as
\[s_f(x) = \mathbf{a}(x)^\intercal \mathbf{K}^{-1} \boldsymbol{\lambda},\]
where $\mathbf{a}(x)$ denotes the vector
\[\mathbf{a}(x) = [A(x,\tau_1) \; \cdots \; A(x,\tau_n)]^{\intercal}.\]
\end{theorem}

\begin{remark}
The combination of Theorem \ref{thm:histopolationunique} with Theorem \ref{thm:AKHSmain} reveals the following fundamental connection between interpolation and histopolation problems: solving a histopolation problem in an AKHS $\mathcal{N}_A(\Omega)$ is equivalent to solving a respective interpolation problem in the RKHS $\mathcal{N}_K(\mathcal{T})$.
\end{remark}

The linear independence of the averaging functionals $\{ \lambda_{\tau_1}, \ldots, \lambda_{\tau_n} \}$ is a fundamental prerequisite for Theorem \ref{thm:histopolationunique}. Since these functionals essentially depend on the domains $\{\omega_{\tau_1}, \ldots, \omega_{\tau_n}\}$, we aim to identify conditions on these domains that guarantee the linear independence of the kernel functions $\{ A(\cdot,\tau_1), \ldots, A(\cdot,\tau_n) \}$. To this end, following Section \ref{sec:AKHSbyRKHS}, we assume now that the averaging kernel $A$ is constructed upon a continuous symmetric positive definite kernel $\Phi$. For such a kernel we can introduce the compact  integral operator $T: L_2(\Omega) \rightarrow \mathcal{N}_{\Phi}(\Omega) \subseteq L_2(\Omega)$, defined by
\begin{equation*}
    T(v)(x) = \int_{\Omega} \Phi(x,y)v(y) \text{d}y, \qquad v \in L_2(\Omega), \quad x \in \Omega.
\end{equation*}
If $ \Omega \subset \mathbb{R}^d $ is compact, the integral operator $T$ maps $L_2(\Omega)$ continuously into the native space $\mathcal{N}_{\Phi}(\Omega)$. It is the adjoint of the embedding operator of the RKHS $\mathcal{N}_{\Phi}(\Omega)$ into $L_2(\Omega)$, i.e., it satisfies
    \begin{equation} \label{eq:adjointembedding}
        \langle f,v \rangle_{L_2(\Omega)} = \langle f,T(v) \rangle_{\mathcal{N}_{\Phi}(\Omega)}.
    \end{equation}
    Furthermore, the range of $T$ is dense in $\mathcal{N}_{\Phi}(\Omega)$, as shown in \citep[Proposition 10.28]{Wendland_2004}.

If the averaging kernel $A$ is constructed upon the positive definite kernel $\Phi$, the corresponding reproducing kernel $K$ is determined by \eqref{eq:reproducingkernel}, and the entries of the histopolation matrix $\mathbf{K}$ in \eqref{eq:histopolationmatrix} based on the finite collection of subsets $\{ \omega_{\tau_1}, \ldots, \omega_{\tau_n}\}$ can be explicitly written as
\begin{equation} \label{eq_definition_matrix_A}
    \mathbf{K}_{i,j} = \frac{1}{|\omega_{\tau_i}|} \frac{1}{|\omega_{\tau_j}|} \int_{\omega_{\tau_i}} \int_{\omega_{\tau_j}} \Phi(x,y)  \text{d}x \text{d}y.
\end{equation}
We can now show the following result, which relates the unisolvence of the histopolation problem with the linear independence of the characteristic functions of the averaging domains in the space $L_2(\Omega)$.

\begin{theorem} \label{thm_A_strictly_positive_definite}
    Let $\Phi \in C(\Omega \times \Omega)$ be a symmetric positive definite kernel on the compact set $\Omega$. We assume further that the RKHS $\mathcal{N}_{\Phi}(\Omega)$ is dense in $L_2(\Omega)$ and that the characteristic functions $\{ \chi_{\omega_{\tau_1}}, \ldots, \chi_{\omega_{\tau_n}} \}$ are linearly independent in $L_2(\Omega)$. Then, the histopolation matrix $\mathbf{K}$ in \eqref{eq_definition_matrix_A} is symmetric and positive definite, and, in particular, invertible.  
\end{theorem}

\begin{proof} 
The symmetry of $\mathbf{K}$ follows from the symmetry of the reproducing kernel $K$, can however be also derived directly from \eqref{eq_definition_matrix_A} by using the symmetry of $\Phi$ and Fubini's theorem.  
To prove that the matrix $\mathbf{K}$ is positive definite, it is convenient to write the corresponding entries in terms of the characteristic functions $\chi_{\omega_{\tau_i}}$ and the integral operator $T$:
\begin{equation*}
\begin{split}
|\omega_{\tau_i}||\omega_{\tau_j}| \mathbf{K}_{i,j} & = \int_{\omega_{\tau_i}} \int_{\omega_{\tau_j}} \Phi(x,y)  \text{d}x \text{d}y = \\
& = \int_{\Omega} \left( \chi_{\omega_{\tau_i}}(y) \int_{\Omega} \Phi(x,y) \chi_{\omega_{\tau_j}}(x)   \text{d}x \right) \text{d}y = \langle \chi_{\omega_{\tau_i}}, T(\chi_{\omega_{\tau_j}}) \rangle_{L_2(\Omega)}.
\end{split}
\end{equation*}
We consider now $\alpha \in \mathbb{R}^n \setminus \{0\}$ and show that $\mathbf{K}$ is positive definite, i.e., that $\sum_{i,j=1}^n \alpha_i \alpha_j \mathbf{K}_{i,j} > 0$. Using the identity above, we get
\begin{equation*}
    \begin{split}
        \sum_{i=1}^n \sum_{j=1}^n \alpha_i \alpha_j \mathbf{K}_{i,j} & = \sum_{i=1}^n \sum_{j=1}^n \alpha_i \alpha_j \left\langle \frac{1}{|\omega_{\tau_i}|}\chi_{\omega_{\tau_i}}, T\left(\frac{1}{|\omega_{\tau_j}|}\chi_{\omega_{\tau_j}}\right) \right\rangle_{L_2(\Omega)} = \\
        & = \sum_{i=1}^n \alpha_i  \left\langle \frac{1}{|\omega_{\tau_i}|}\chi_{\omega_{\tau_i}}, T\left( \sum_{j=1}^n  \frac{\alpha_j}{|\omega_{\tau_j}|}\chi_{\omega_{\tau_j}}\right) \right\rangle_{L_2(\Omega)} = \\
        & = \left\langle \sum_{i=1}^n  \frac{\alpha_i}{|\omega_{\tau_i}|}\chi_{\omega_{\tau_i}}, T\left( \sum_{j=1}^n  \frac{\alpha_j}{|\omega_{\tau_j}|}\chi_{\omega_{\tau_j}}\right) \right\rangle_{L_2(\Omega)} = \langle v, T(v) \rangle_{L_2(\Omega)},
    \end{split}
\end{equation*}
where $v = \sum_{i=1}^n  \frac{\alpha_i}{|\omega_{\tau_i}|} \chi_{\omega_{\tau_i}} $ and $v \neq 0$ because $\{\omega_{\tau_1}, \ldots, \omega_{\tau_n}\}$ are linearly independent. By the identity \eqref{eq:adjointembedding} shown in \citep[Proposition 10.28]{Wendland_2004}, 
we therefore get 
\begin{equation*}
    \begin{split}
        \sum_{i=1}^n \sum_{j=1}^n \alpha_i \alpha_j \mathbf{K}_{i,j}  = \langle v, T(v) \rangle_{L_2(\Omega)} = \langle T(v), T(v) \rangle_{\mathcal{N}_{\Phi}(\Omega)} = \| T(v) \|_{\mathcal{N}_{\Phi}(\Omega)}^2 > 0,
    \end{split}
\end{equation*}
where we have strict positivity due to the injectivity of the integral operator $T$. The injectivity of $T$ follows from the fact that the range of the adjoint operator (the embedding operator from $\mathcal{N}_{\Phi}(\Omega)$ into $L_2(\Omega)$) corresponds to $\mathcal{N}_{\Phi}(\Omega)$ and the assumption that $\mathcal{N}_{\Phi}(\Omega)$ is dense in $L_2(\Omega)$.
\end{proof}

Theorem \ref{thm_A_strictly_positive_definite} is quite valuable for obtaining theoretical guarantees of the invertibility of the histopolation matrix $\mathbf{K}$. In fact, for many well-known positive definite kernels, such as the Mat{\'e}rn and Wendland kernels \citep[cf.][]{Wendland_2004}, it is known that the associated RKHS's are equivalent to Sobolev spaces. Therefore, their restrictions to a domain $\Omega$ are dense in $L_2(\Omega)$. In such cases it is then sufficient to verify that the characteristic functions $\{ \chi_{\omega_{\tau_1}}, \ldots, \chi_{\omega_{\tau_n}} \}$ are linearly independent in $L_2(\Omega)$. 

In the following, we establish a sufficient condition for the linear independence of the characteristic functions $\{ \chi_{\omega_{\tau_1}}, \ldots, \chi_{\omega_{\tau_n}} \}$ in $L_2(\Omega)$. This result  can be seen as a minor extension of the criterion presented in \citep[Theorem 4.7]{Albrecht2024}, where linear independence was considered in the space of compactly supported distributions. Compared to the distributional setting, the Hilbert space setting in $L_2(\Omega)$ enables a simpler proof of the criterion in terms of a Gram matrix argument.  

\begin{theorem} \label{thm_linear_independence_domains}
    Let $\{\omega_{\tau_1}, \ldots, \omega_{\tau_n}\}$ be a collection of bounded domains in $\Omega \subseteq \mathbb{R}^d$. Assume that they are ordered in such a way that there exist open balls $o_{\varepsilon}^{(i)}$, $i \in \{1, \ldots, n\}$, with radius $\varepsilon>0$ such that
    \begin{equation} \label{eq:criterionindependence}
    o_{\varepsilon}^{(i)} \subseteq \omega_{\tau_i}, \quad o_{\varepsilon}^{(i)} \cap \bigcup_{j = i+1}^n \omega_{\tau_j} = \emptyset, \quad i \in \{1, \ldots, n\}. 
    \end{equation}
    Then, the characteristic functions $\{ \chi_{\omega_{\tau_1}}, \ldots, \chi_{\omega_{\tau_n}} \}$ are linearly independent in $L_2(\Omega)$.   
\end{theorem}

\begin{proof}
Because of \eqref{eq:criterionindependence}, the Gram matrix $\mathbf{G}$ with the entries $\mathbf{G}_{i,j} = \langle \chi_{o_{\varepsilon}^{(i)}}, \chi_{o_{\varepsilon}^{(j)}} \rangle_{L_2(\Omega)}$ is a diagonal matrix with positive diagonal entries. For this the characteristic functions $\{\chi_{o_{\varepsilon}^{(1)}}, \ldots, \chi_{o_{\varepsilon}^{(n)}}\}$ are orthogonal, and thus also linearly independent in $L_2(\Omega)$. 

Consider now a second Gram matrix $\mathbf{H}$ based on the inner products
$\mathbf{H}_{i,j} = \langle \chi_{o_{\varepsilon}^{(i)}}, \chi_{\omega_{\tau_j}}  \rangle_{L_2(\Omega)}$ 
between the linearly independent functions $\{\chi_{o_{\varepsilon}^{(1)}}, \ldots, \chi_{o_{\varepsilon}^{(n)}}\}$ and the system $\{ \chi_{\omega_{\tau_1}}, \ldots, \chi_{\omega_{\tau_n}} \}$. Because of \eqref{eq:criterionindependence}, the matrix $\mathbf{H}$ is lower triangular and has the same diagonal
as $\mathbf{G}$. Therefore, also the Gram matrix $\mathbf{H}$ is invertible, implying that the characteristic functions $\{ \chi_{\omega_{\tau_1}}, \ldots, \chi_{\omega_{\tau_n}} \}$ are linearly independent in $L_2(\Omega)$. 
\end{proof}

\begin{example}
On a compact interval $I \subseteq \mathbb{R}$, consider the subintervals $\{\omega_{y_1}, \ldots, \omega_{y_n}\}$ given as $\omega_{y_{i}} = [-a/2,a/2] + y_i$ with ordered centers $y_1 < y_2 < \cdots < y_n$ and uniform interval length $a > 0$. Then, by Theorem \ref{thm_linear_independence_domains}, the characteristic functions $\{ \chi_{\omega_{y_1}}, \ldots, \chi_{\omega_{y_n}} \}$ are linearly independent in $L_2(I)$.
\begin{itemize}
    \item[$(i)$] The Mat{\'e}rn kernel function $\phi(x) = e^{- \lambda |x|}$ considered in Example \ref{ex:RUAK} (i) generates a RKHS which is norm-equivalent to the Sobolev space $H^1(I)$. Therefore, $\mathcal{N}_{\phi}(I)$ is dense in $L_2(I)$. For the segments introduced above, we can therefore apply Theorem \ref{thm_A_strictly_positive_definite} and obtain that the respective histopolation matrix $\mathbf{K}$ is invertible. This guarantees the unisolvence of the histopolation problem in the example illustrated in Fig. \ref{fig:histopolationmatern}.
    \item[$(ii)$] For the Hilbert space $L_2(I)$ we know directly that the indicator functions $\{ \chi_{\omega_{y_1}}, \ldots, \chi_{\omega_{y_n}} \}$ are linearly independent. For this, also the averaging kernel $A_1(x,y) = \frac{1}{a} \chi_{\omega_y}(x)$ introduced in \eqref{eq:indicatorkernel1D} provides linearly independent basis functions $\{A_1(x,y_1), \ldots A_1(x,y_n)\}$ and the corresponding histopolation matrix $\mathbf{K}_1$ is invertible. This ensures the unisolvence of the histopolation problem shown in Fig. \ref{fig:histopolationindicator}.
\end{itemize}
\end{example}

\section{Uniform error estimates for average values} \label{sec:errorestimates}

In general, uniform convergence of the histopolant $s_f$ towards a function $f$ cannot be guaranteed without additional assumptions on the averaging domains. While this is not unexpected for discontinuous functions, it can be also observed for smooth functions, as shown in the numerical experiments of Section \ref{sec:numericalexperiments}. For histopolation, a more natural type of convergence is the uniform convergence of the mean values. In fact, the isometry of Theorem \ref{thm:AKHSmain} will guarantee that for a large class of functions the average values $\lambda_\tau(s_f)$ converge towards the values $\lambda_{\tau}(f)$ if the domains $\omega_\tau$ are sufficiently dense. 

To obtain respective error bounds, we introduce the Lagrange basis functions $\ell_j(x)$, $j \in \{1, \ldots,n\}$ as the elements in the space $\mathcal{N}_{A}(\Omega,\mathcal{T}_n)$ that satisfy the cardinal histopolation conditions
\[ \lambda_{\tau_i} (\ell_j) = \left\{ \begin{array}{ll}
     1 & \text{if}\; i = j,  \\
     0 & \text{if}\; i \neq j.
\end{array}\right. \]

\begin{figure}[h]
    \centering
    \includegraphics[width=0.49\linewidth]{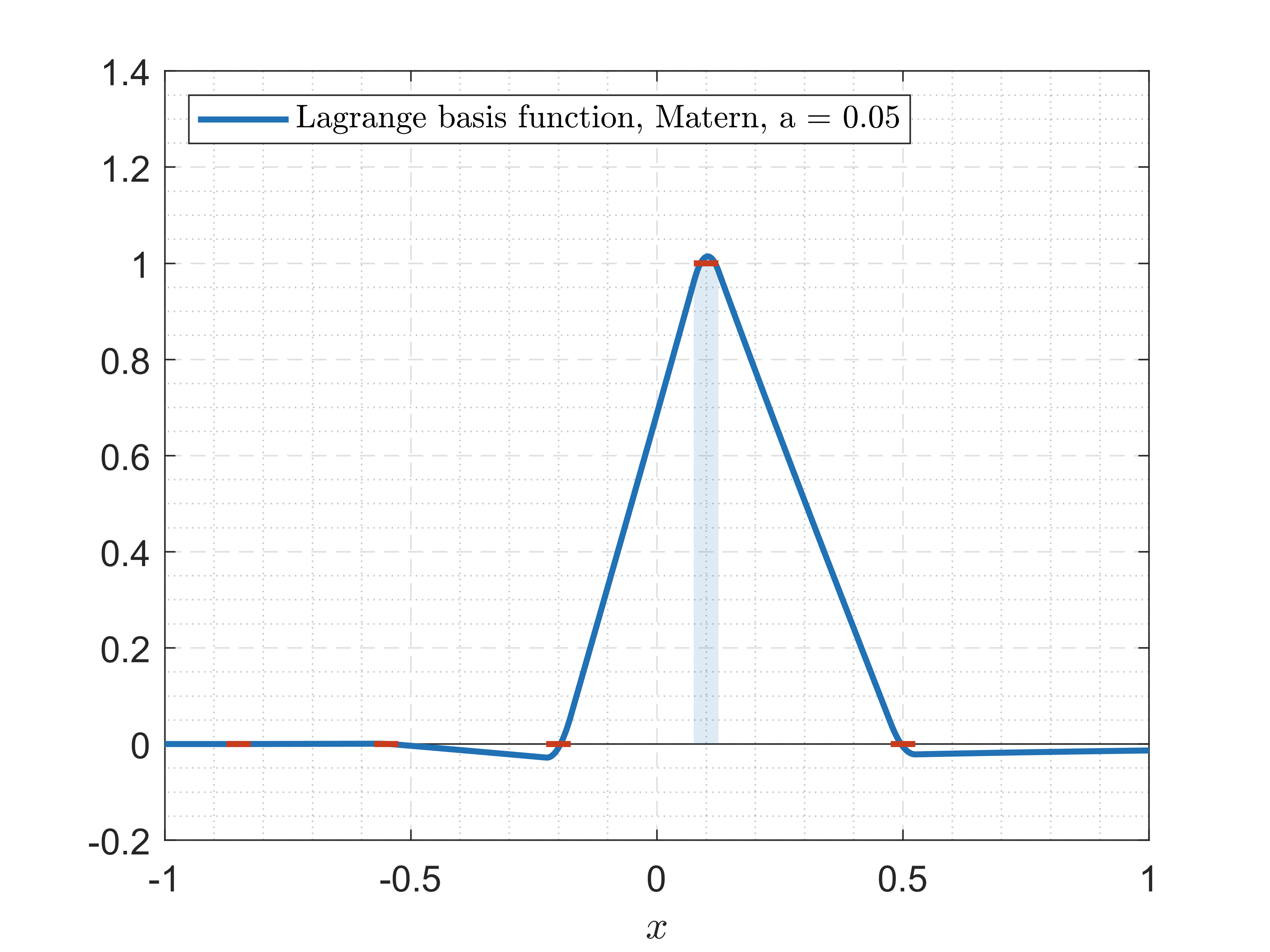}
    \includegraphics[width=0.49\linewidth]{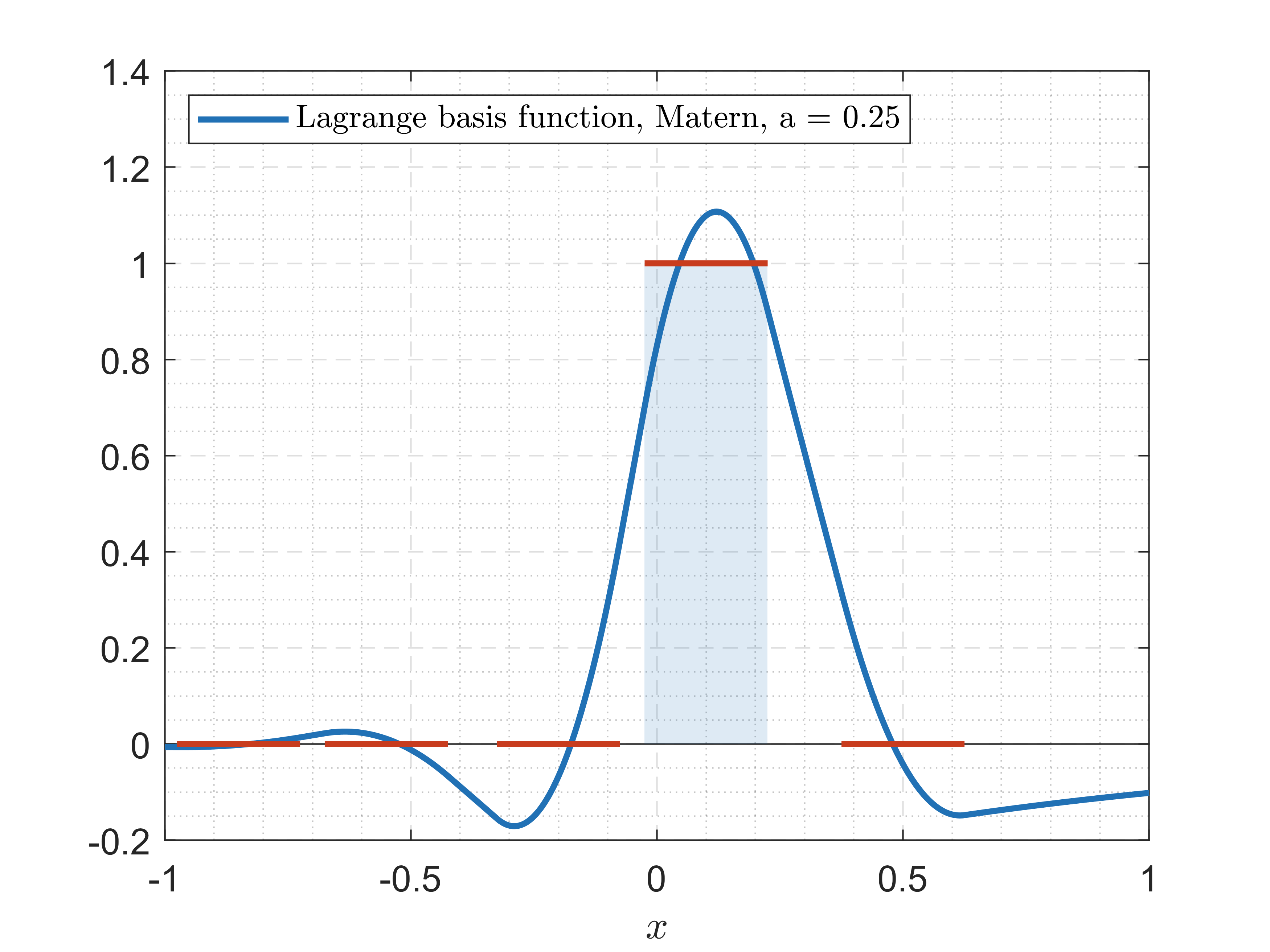}
    \caption{Lagrange basis functions for histopolation on $5$ segments (red) with length $a = 0.05$ (left) and $a = 0.25$ (right) using the one-dimensional averaged Mat{\'e}rn function in Example \ref{ex:RUAK} (i).}
    \label{fig:cardinalbasisfunctions}
\end{figure}

If the conditions of Theorem \ref{thm:histopolationunique} are satisfied, the functions $\ell_j(x)$, $j \in \{1, \ldots, n\}$, are uniquely determined and form a basis of the $n$-dimensional space $\mathcal{N}_{A}(\Omega,\mathcal{T}_n)$. Furthermore, we can write the histopolant $s_f$ in the cardinal form
\begin{equation} \label{eq:interpolatorLBF}
    s_f(x) = \sum_{j = 1}^n \lambda_{\tau_j}(f) \ell_j(x). 
\end{equation}
We can transfer this basis and the histopolant $s_f$ into the $n$-dimensional subspace 
\[ \mathcal{N}_K(\mathcal{T},\mathcal{T}_n)= \mathrm{span} \{K(\cdot, \tau_i) \ : \ i \in \{1, \ldots, n\} \} \]
of the RKHS $\mathcal{N}_K(\mathcal{T})$ by using the isometry $ \Lambda: \mathcal{N}_A (\Omega) \to \mathcal{N}_K(\mathcal{T}) $ introduced in Theorem \ref{thm:AKHSmain}. The respective (nodal) Lagrange basis of the space $\mathcal{N}_K(\mathcal{T},\mathcal{T}_n)$ is given as $\{\Lambda \ell_i(\tau) \; : \; i \in \{1, \ldots, n\}\}$ and the mapped histopolant in his cardinal form reads as 
\[ \Lambda s_f (\tau) = \sum_{j = 1}^n \lambda_{\tau_j}(f) \Lambda \ell_j(\tau). \]
Note that if the given function $f$ lies in the AKHS $\mathcal{N}_A(\Omega)$, we could as well write $\Lambda f(\tau_j)$ instead of $\lambda_{\tau_j}(f)$. 
For the reproducing kernel $K$, we can define the power function $P_{K,\mathcal{T}_n}(\tau)$ as
\[P_{K,\mathcal{T}_n}(\tau) := \left\| K(\cdot, \tau) - \sum_{j=1}^n \lambda_{\tau}(\ell_j) K(\cdot, \tau_{j}) \right\|_{\mathcal{N}_K(\mathcal{T})} = \left\| A(\cdot, \tau) - \sum_{j=1}^n \lambda_{\tau}(\ell_j) A(\cdot, \tau_{j}) \right\|_{\mathcal{H}},\]
where the second equality follows again by the isometry between the spaces $\mathcal{N}_A(\Omega)$ and $\mathcal{N}_K(\mathcal{T})$. For this power function, we can obviously also write
\begin{align*}
(P_{K,\mathcal{T}_n}(\tau))^2 &= \left\| K(\cdot, \tau) - \sum_{j=1}^n \lambda_{\tau}(\ell_j) K(\cdot, \tau_{j}) \right\|_{\mathcal{N}_K(\mathcal{T})}^2 \\
&= K(\tau,\tau) - 2 \sum_{j = 1}^n \lambda_{\tau}(\ell_j) K(\tau,\tau_j) + \sum_{i = 1}^n \sum_{j = 1}^n \lambda_{\tau}(\ell_i) \lambda_{\tau}(\ell_j) K(\tau_i,\tau_j).
\end{align*}
For RKHS's, the power function is a useful tool to obtain generic error estimates. The isometry in Theorem \ref{thm:AKHSmain} allows to transfer these estimates also for AKHS's. 

\begin{corollary} \label{cor:errorestimate}
Assume that the averaging functionals $\{ \lambda_{\tau} \ : \ \tau \in \mathcal{T}\}$ are continuous in $\mathcal{H}$ and generate an AKHS $\mathcal{N}_A(\Omega)$. Further, let $f \in \mathcal{N}_A(\Omega)$, $\mathcal{T}_n = \{\tau_1, \ldots, \tau_n\} \subseteq \mathcal{T}$, and $s_f$ the histopolant calculated upon the linearly independent average functionals $\lambda_{\tau_i}(f)$ based on the domains $\omega_{\tau_i}$, $i \in \{1, \ldots, n\}$. Then, we have
\[ \sup_{\tau \in \mathcal{T}}|\lambda_{\tau}(f - s_f)| \leq \sup_{\tau \in \mathcal{T}} P_{K,\mathcal{T}_n}(\tau) \, \|f\|_{\mathcal{H}}.\]
\end{corollary}

\begin{proof} Using the averaging property of the kernel $A(x,\tau)$ in the AKHS $\mathcal{N}_A(\Omega)$, and the Cauchy-Schwarz inequality, we directly get
\begin{align*}
|\lambda_{\tau}(f - s_f)| = \left| \left\langle f, A(\cdot, \tau) - \sum_{j=1}^n \lambda_{\tau}(\ell_j) A(\cdot, \tau_{j}) \right\rangle_{\mathcal{H}} \right| \leq \|f\|_{\mathcal{H}} P_{K,\mathcal{T}_n}(\tau).
\end{align*}
Passing to the supremum on $ \tau \in \mathcal{T} $ we get the claim.
\end{proof}

Corollary \ref{cor:errorestimate} enables us to leverage the extensive literature on (nodal) kernel interpolation and to exploit known estimates of the power function. 
In this context, uniform error estimates of the error $f - s_f$ are typically achieved by bounding the power function in terms of the fill distance of the set of interpolation nodes. For uniform radial averaging kernels $A(x,y) = \alpha(\|x-y\|_2)$, similar estimates for the power function $P_{K,\mathcal{T}_n}(y)$ of the reproducing kernel $K$ can be achieved using the fill distance
\[
h = h_{\Omega,\mathcal{T}_n} := \sup_{y \in \Omega} \; \min_{i \in \{1, \ldots, n\}} \| y - y_i \|_2.
\]
of the centers $y_i$ of the domains $\omega_{y_i} = \omega + y_i$. For radial kernels $K$ typical estimates of the power function lead to bounds of the form
\[ P_{K,\mathcal{T}_n}(y)^2 \leq F(h),\]
with an increasing function $F:[0,\infty) \to [0,\infty)$ satisfying $F(0) = 0$  \citep{Schaback1995}.

\begin{example}
We consider as a radial univariate kernel the Mexican hat wavelet introduced in Example \ref{ex:RUAK} (iv). The corresponding reproducing kernel $K$ is determined in terms of a linear combination $\kappa(x) = \frac{\lambda}{2 a^2} (2 e^{- \frac{x^2}{\lambda}} - e^{- \frac{(x+a)^2}{\lambda}} - e^{- \frac{(x-a)^2}{\lambda}} )$ of three Gaussians. By \citep{MadychNelson1992}, we know that the respective power function can be bounded by $P_{K,\mathcal{T}_n}(y) \leq e^{-\frac{\delta}{h}}$ for some $\delta > 0$. Therefore, Corollary \ref{cor:errorestimate} leads to the error estimate
\[ \sup_{y \in \Omega} \left| \frac{1}{a}\int_{y-a/2}^{y + a/2} (f(x) - s_f(x)) \de x \right| \leq \|f\|_{\mathcal{N}_{\Phi}(\mathbb{R})} e^{-\frac{\delta}{h}}\]
for histopolation with the averaging kernel constructed upon the Mexican hat wavelet $\phi$ and a set $\mathcal{T}_n$ of $n$ uniform domains with centers $y_i$ in $\Omega$ and fill distance $h$ between the centers.  
\end{example}

\section{Approximate averages by quadrature formulas} \label{sect:implementationstrategy}

To set up the linear system \eqref{eq:mainhistopolationsystem} for the numerical calculation of the kernel histopolant $s_f$, we have to know or determine the matrix entries $\mathbf{K}_{i,j}$ of $\mathbf{K} $ in \eqref{eq:histopolationmatrix}, as well as the entries of the right hand side in \eqref{eq:rhs}. In this way, we are able to compute the expansion coefficients of the kernel histopolant $s_f$ in \eqref{eq:approxAKHS}, as well as to determine $s_f$ itself. 

While the average values of the function $f$ over the domains $\omega_{\tau_i}$ are typically given as initial data, the average values $K(\tau_i,\tau_j)$ are only known in particular cases, some of which we saw in the previous sections. To calculate the kernel entries $K(\tau_i,\tau_j)$ for general kernels $\Phi(x,y)$ and general domains $\omega_{\tau_i}$, we require quadrature or cubature formulas that approximate the integral values by a weighted sum of function evaluations. We can write such a quadrature formula as

\begin{align*} 
\mathbf{K}_{i,j} & = \frac{1}{|\omega_{\tau_i}|} \frac{1}{|\omega_{\tau_j}|} \int_{\omega_{\tau_i}} \left(\int_{\omega_{\tau_j}} \phi(\|x - y\|_2) \mathrm{d} x \right) \mathrm{d} y \approx  \sum_{\ell = 1}^{m}  w_{\ell}^{(j)} \sum_{k=1} ^{m} w_k^{(i)} \phi(\|x_{k} - x_{\ell}\|_2),
\end{align*}
which gives the approximated matrix entries 
\begin{equation} \label{eq:approxMatrix}
\mathbf{K}_{i,j}^Q =  (\mathbf{w}^{(j)})^T \mathbf{\Phi}^Q \mathbf{w}^{(i)} ,
\end{equation}
where $\mathbf{w}^{(i)} = (w_1^{(i)} \, \cdots \, w_{m}^{(i)})^T$ denotes the (column) vector of the non-negative quadrature weights $w_k^{(i)} \geq 0$ used for the domain $\omega_{\tau_i}$, and $\{x_1, \ldots, x_{m}\}$ the quadrature nodes distributed over the union of the $n$ domains. The matrix $\mathbf{\Phi}^Q \in \mathbb{R}^{m \times m}$ contains the entries 
$$\mathbf{\Phi}_{k,\ell}^{Q} =  \phi(\|x_k - x_\ell \|_2) = \Phi(x_k,x_l). $$ 
We saw already that the histopolation matrix $ \mathbf{K} $ is symmetric and positive definite, when computed on a continuous level with suitable averaging domains. In a similar way, such a property can be stated also for the approximated histopolation matrix $ \mathbf{K}^Q $.

\begin{lemma}
    Let $ \Phi(x, y)$ be a symmetric and positive semi-definite kernel on $\Omega$. Then, the matrix $\mathbf{K}^Q $ is symmetric and positive semi-definite.
    Furthermore, if $\Phi$ is positive definite and for each domain $\omega_{\tau_i}$ there is a node $x_k$ such that $w_k^{(i)} > 0$ and $w_k^{(j)} = 0$ for all other domains $\omega_{\tau_j}$, $j \neq i$, then the matrix $\mathbf{K}^Q $ is positive definite.  
\end{lemma}

\begin{proof} 
The symmetry of $ \mathbf{K}^Q  $ is a direct consequence of the symmetry of the kernel $ \Phi $. Namely, we have
$$ \mathbf{\Phi}_{k,\ell}^{Q} = \Phi (x_k, x_\ell) = \Phi(x_\ell, x_k ) = \mathbf{\Phi}_{\ell,k }^{Q}, $$
and to obtain the claim, it is sufficient to plug this identity in \eqref{eq:approxMatrix}. Also the positive semi-definiteness follows from the positive semi-definiteness of the kernel $\Phi$: based on the definition of the matrix entries $ \mathbf{K}_{i,j}^Q $ in \eqref{eq:approxMatrix}, we get for $v \in \mathbb{R}^n$ the inequality
\begin{align*}
    v^T \mathbf{K}^Q v & = v^T \begin{pmatrix}
        \mathbf{K}_{1,1}^Q & \ldots & \mathbf{K}_{1,n}^Q \\
        \vdots & & \vdots \\
       \mathbf{K}_{n,1}^Q & \ldots & \mathbf{K}_{n,n}^Q
    \end{pmatrix} v  = v^T \begin{pmatrix}
        (\mathbf{w}^{(1)})^T \\ \vdots\\ (\mathbf{w}^{(n)})^T
    \end{pmatrix} \mathbf{\Phi}^Q \begin{pmatrix}
        \mathbf{w}^{(1)} \cdots  \mathbf{w}^{(n)}
    \end{pmatrix} v  \geq 0.
\end{align*}
The fact that for every domain $\omega_{\tau_i}$ there is a node $x_k$ such that $w_k^{(i)} > 0$ and $w_k^{(j)} = 0$ for all other domains implies that the $n$ vectors $\{\mathbf{w}^{(1)}, \ldots, \mathbf{w}^{(n)}\}$ are linearly independent. Therefore, if the kernel $\Phi$ is positive definite, then the matrix $\mathbf{\Phi}^Q$ is positive definite, and consequently so is also $\mathbf{K}^Q$.
\end{proof}

Once the approximate histopolation matrix $\mathbf{K}^Q$ is calculated, the respective expansion coefficients can be obtained as $ \mathbf{c}^Q = (\mathbf{K}^Q)^{-1} \boldsymbol{\lambda} $. Then, the quadrature analog of the approximant in \eqref{eq:approxAKHS} is given as
\begin{align*}
    s_f^Q(x) &= \sum_{j = 1}^n c_j^Q \sum_{k = 1}^m w_k^{(j)}\phi(\Vert x - x_k \Vert_2) 
      \approx \sum_{j = 1}^n c_j  \frac{1}{|\omega_{\tau_j}|}\int_{\omega_{\tau_j}} \phi(\Vert x - y \Vert_2) \de y = \sum_{j = 1}^n c_j A (x, \tau_j) = s_f(x)
\end{align*}

The analysis of quadrature rules for RBF's is a very active field of research, we refer to \citep{GlaubitzQuad,GlaubitzIMA} for some recent advances on this topic.  

\section{Numerical Experiments} \label{sec:numericalexperiments}

We now present some numerical tests that show the applicability of kernel histopolation and provide insights into its convergence behavior. For simplicity, we restrict ourselves to one- and two-dimensional scenarios. In these tests, we do not optimize the shapes of the kernels, although we note that procedures such as the one described in \citep{FasshauerOptimal} can be applied with only minor modifications.

\subsection{Convergence in the univariate case: $ d = 1 $}

\begin{figure}[htbp]
    \centering
    \includegraphics[width=0.45\linewidth]{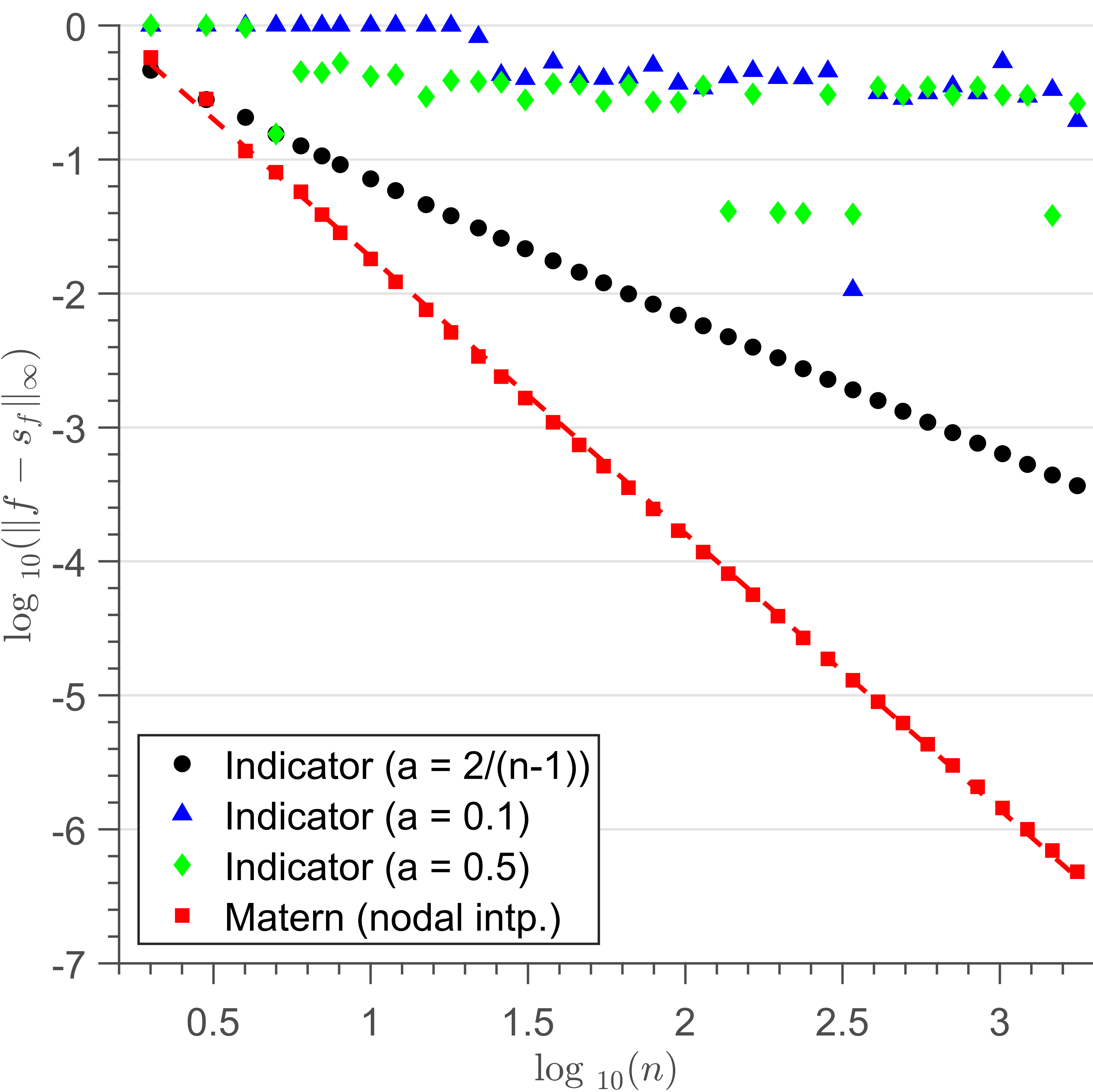} \qquad
    \includegraphics[width=0.45\linewidth]{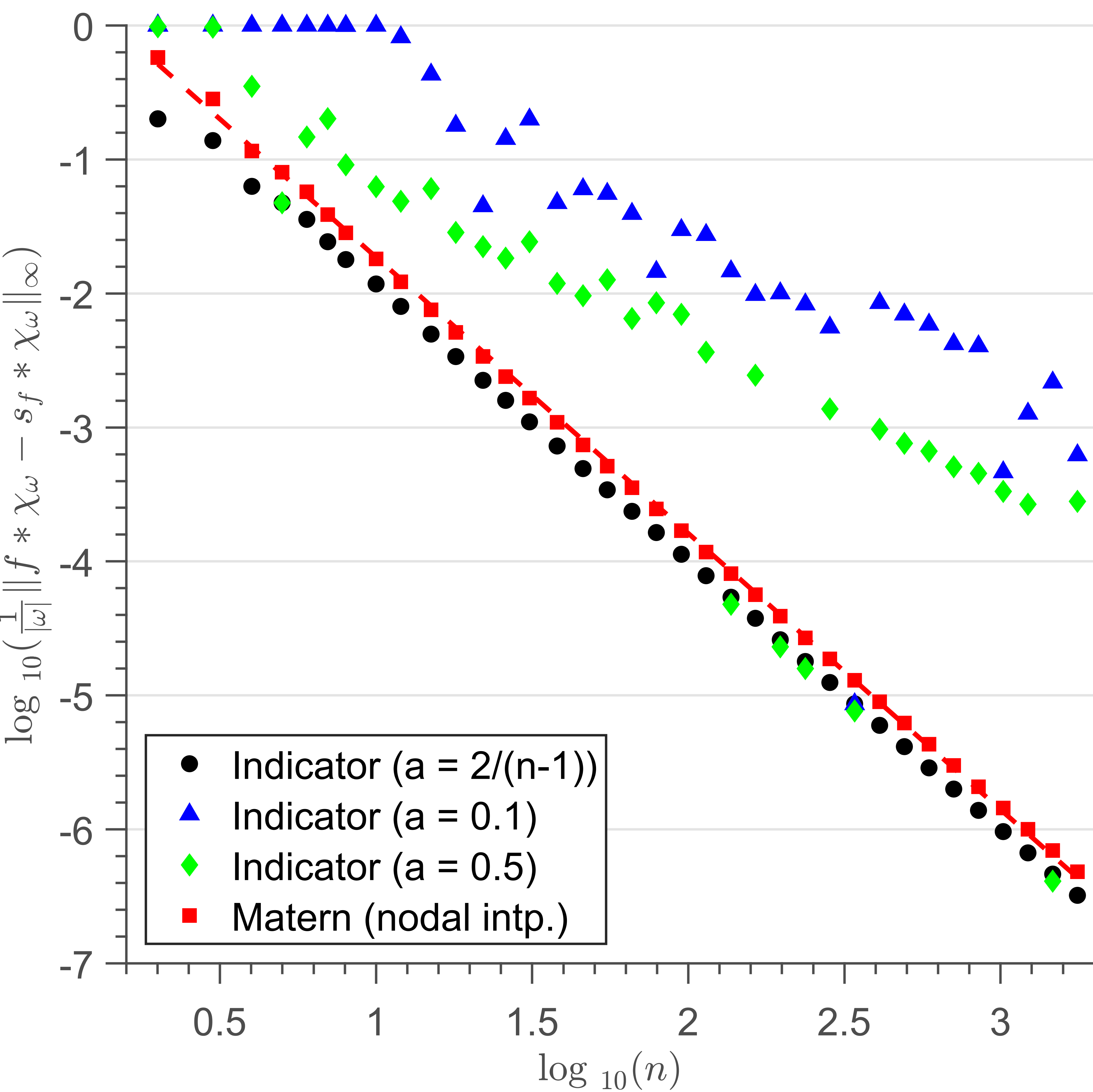} 
    \caption{Uniform errors (left) and uniform mean errors (right) of histopolation with the indicator kernel compared to the errors of the Mat{\'e}rn interpolation. While we do not have uniform convergence for fixed interval length, the errors of the mean values do converge uniformly also in this case. }
    \label{fig:convergenceindicator}
\end{figure}

\begin{figure}[htbp]
    \centering
    \includegraphics[width=0.45\linewidth]{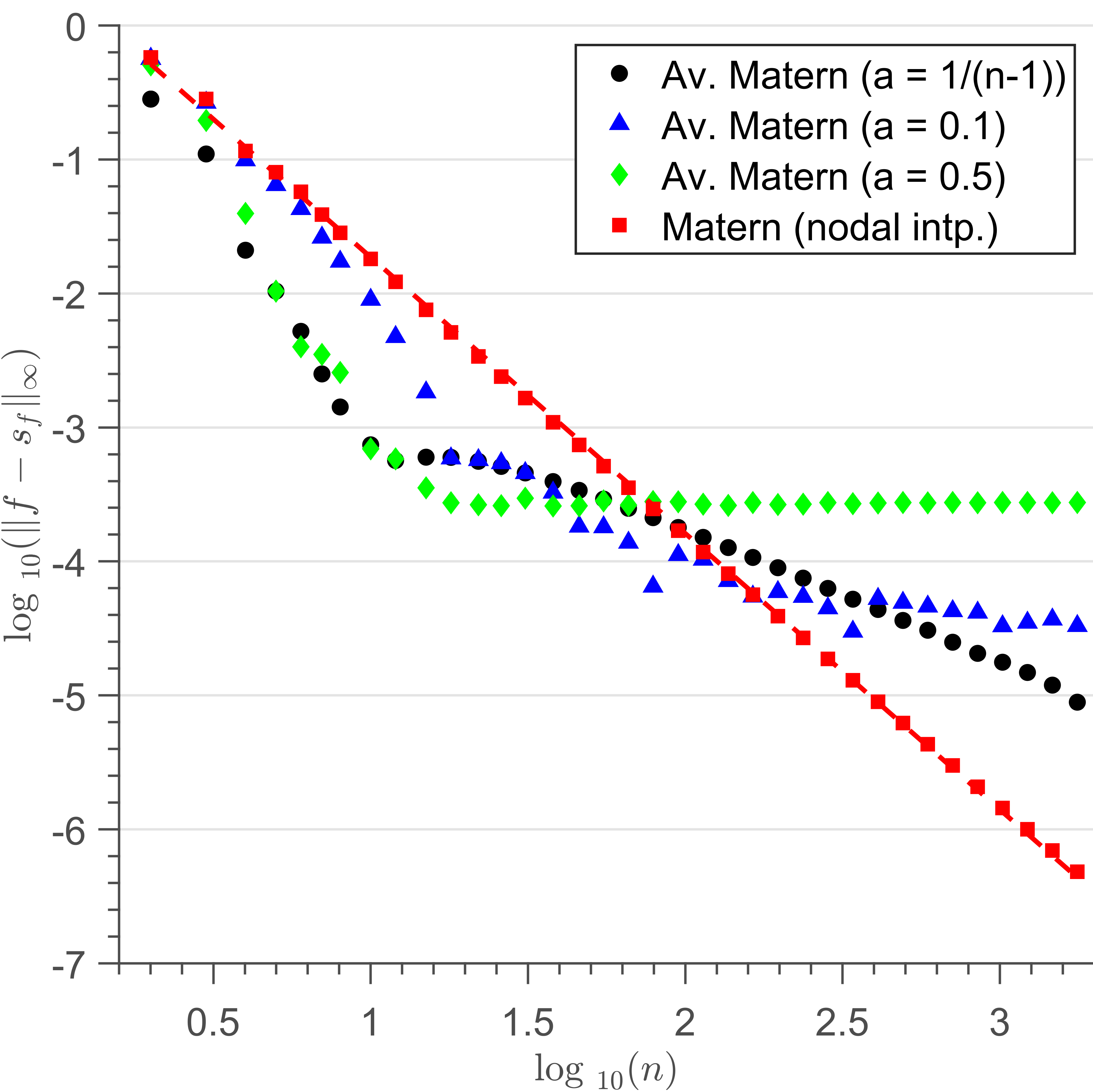} \qquad
    \includegraphics[width=0.45\linewidth]{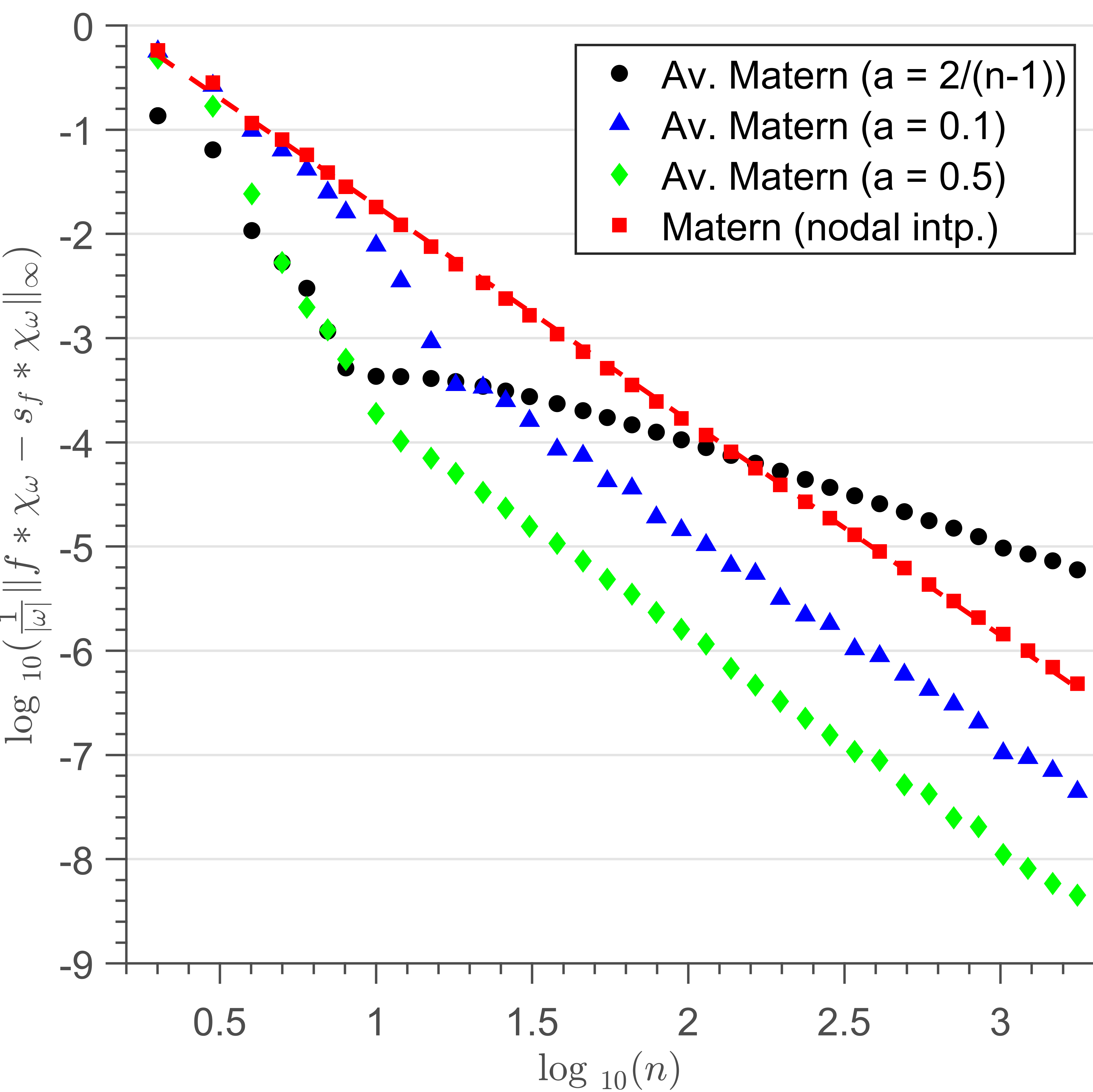} 
    \caption{Uniform errors (left) and uniform mean errors (right) of histopolation with the averaged Mat{\'e}rn kernel. While we do not get convergence in the uniform norm for fixed interval length, the respective error of the mean values converges uniformly with the same rate as for the Mat{\'e}rn interpolant. }
    \label{fig:convergencematern}
\end{figure}

We start with a 1D example and analyze the convergence of kernel histopolation on the interval $\Omega = [-1,1]$ using average values of the test function $ f(x) = \frac{1}{1+(x-0.4)^2}$. As averaging domains, we take $n$ equidistant centers $y_i = -1 + 2 \frac{i-1}{n-1} $ in $[-1,1]$ and consider the segments $\omega_{y_i} = [y_i-\frac{a}{2}, y_i + \frac{a}{2}]$ for different interval lengths $a$. To compare different interval lengths, we will use values $a = \frac{2}{n-1}$ that decrease with the number of data, as well as constant values $a = 0.1$ and $a = 0.5$. As averaging kernels, we will use the non-continuous indicator kernel introduced in \eqref{eq:indicatorkernel1D}, as well as the averaged Mat{\'e}rn kernel given in Example \ref{ex:RUAK} (i). The corresponding histopolants are calculated by summation of the basis kernel functions in formula \eqref{eq:approxAKHS}, where the expansion coefficients are obtained by the solution of the linear system \eqref{eq:mainhistopolationsystem}. The uniform norms 
$$ \|f - s_f\|_{\infty}\quad  \text{and}  \quad \left\|  \lambda_y (f - s_f) \right\|_{\infty} = \left\|  \frac{1}{a} \int_{y-\frac{a}{2}}^{y + \frac{a}{2}} (f(x) - s_f(x)) \de x \right\|_{\infty} $$
of the error $f - s_f$ and the averaged error on $\Omega = [-1,1]$ are visulized in Fig. \ref{fig:convergenceindicator} (for the indicator kernel) and in Fig. \ref{fig:convergencematern} (for the averaged Mat{\'e}rn kernel).

The results shown in the two figures shed some light on the convergence of histopolation. If the interval length $a$ is kept fixed, Fig. \ref{fig:convergenceindicator} and \ref{fig:convergencematern} show that for both kernels there is no uniform convergence, but the errors averaged over intervals of length $a$ do converge to zero. This aligns with the behavior illustrated in Fig.~\ref{fig:histopolationindicator}, where isolated peaks of the histopolant emerge and obstruct uniform convergence. Similar peaks can be observed also for the averaged Mat{\'e}rn kernel, but only at a lower error magnitude. The increased smoothness of the averaged Mat{\'e}rn kernel compared with the indicator kernel is reflected in its improved convergence rates. 

Fig.~\ref{fig:convergenceindicator} and \ref{fig:convergencematern} also display an interesting convergence phenomenon in the case that the interval lengths $a = \frac{2}{n-1}$ decrease with the number $n$ of segments. For the indicator kernel, the histopolant gives a piecewise constant approximation of $f$, yielding a linear rate of convergence of order $\frac{1}{n}$ in the uniform norm and of order $\frac{1}{n^2}$ for the uniform mean errors. Similarly, we observe convergence also for the averaged Mat{\'e}rn kernel if the segment length $a = \frac{2}{n-1}$ decresases. While in the uniform norm we get convergence of order $\frac{1}{n}$, this rate does not appear to improve when considering mean errors. In fact, in this case convergence becomes worse compared with scenarios of fixed interval length $a$. This particular phenomenon certainly deserves a more thorough investigation from a theoretical point of view.

\subsection{A 2D application to biomedical imaging}

A natural application of two-dimensional histopolation arises in imaging. Each pixel of a grayscale image can be interpreted as the average value of a bivariate function over a small square region. From this perspective, possible applications of histopolation are image upscaling or compression.

As an example, one may compress an image by combining neighboring pixels to a larger pixel and using the average value as a new pixel information (this is also known as pixel binning). Afterwards, histopolation can be applied to the resulting low-resolution object, re-obtaining a new image with the old dimensions. 
In Fig. \ref{fig:moon}, we applied this procedure to a $256 \times 256$ image and compared the original to the reconstructions obtained by histopolation. Due to the rectangular alignment of the pixels, we used the tensor-product of two univariate averaging Mat{\'e}rn kernels with parameter $\lambda = 1$ as underlying bivariate histopolation kernel.   

\begin{figure}
\centering
    \includegraphics[width = 0.96\textwidth]{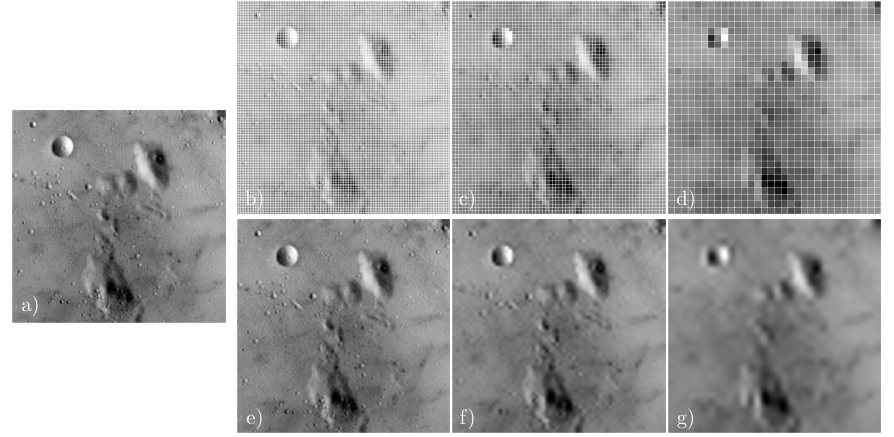} \,
\caption{a) Original image of size $ 256 \times 256 $; b)c)d) reduced image information based on $ 128 \times 128 $, $ 64 \times 64 $ and $ 32 \times 32 $ average pixel values; e)f)g) respective reconstructions based on kernel histopolation using the product Mat{\'e}rn kernel with parameter $\lambda = 1$.}
\label{fig:moon}
\end{figure}

Particularly in biomedical imaging, histopolation techniques prove to be effective. In Fig. \ref{fig:MRI}, we apply the histopolation algorithm to upscale an fMRI image of size $140 \times 140$, using the pixel values as averaging information. The resulting upscaled images are of size $280 \times 280$ and $420 \times 420$. As in the previous example, the histopolant is calculated using an averaged product Mat{\'e}rn kernel with $\lambda = 1$.  

\begin{figure}[h!]
\centering
    \includegraphics[width = 0.96\textwidth]{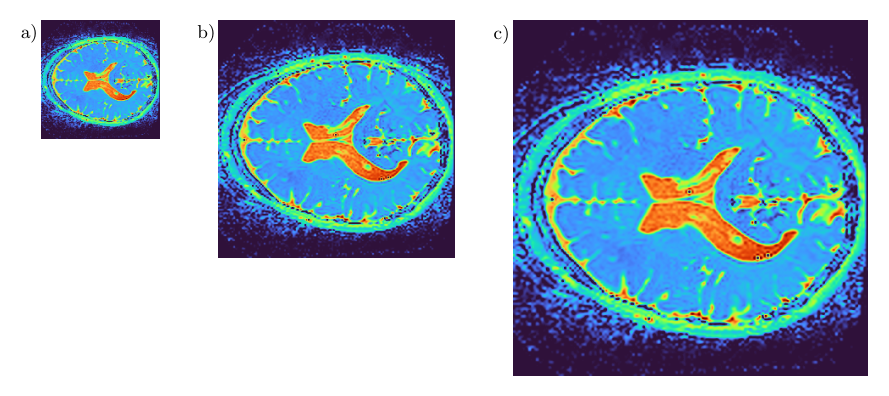} 
\caption{a) Original fMRI image of size $140 \times 140$, b) c) upscaled images of size $280 \times 280$ and $420 \times 420$ pixels using histopolation with the averaged product Mat{\'e}rn kernel and parameter $\lambda = 1$.}
\label{fig:MRI}
\end{figure}

\section*{Acknowledgements}
This project has been funded by the European Union – NextGenerationEU under the National Recovery and Resilience Plan (NRRP), Mission 4 Component 2 Investment 1.1 - Call PRIN 2022 No. 104 of February 2, 2022 of Italian Ministry of University and Research; Project 2022FHCNY3 (subject area: PE - Physical Sciences and Engineering) ``Computational mEthods for Medical Imaging (CEMI)''. Further support by the Indam research group GNCS (Indam-GNCS project CUP\_E53C24001950001) and the Italian Mathematical Union (research group UMI-TAA) is gratefully acknowledged.

\bibliographystyle{IMANUM-BIB}
\bibliography{bibkernel}

\end{document}